\documentclass[a4paper]{amsart}

\usepackage[mathscr]{eucal}
\usepackage{amsmath,amssymb,amsbsy,comment}
\usepackage[all]{xy}
\usepackage[abs]{overpic}
\usepackage{caption}
\usepackage{hyperref}
\hypersetup{
  colorlinks = true,
  citecolor = green,
  urlcolor = black,
  pdfpagemode = UseNone,
  bookmarksnumbered=true,
  bookmarks=true
}
\title[Computation of annular capacity by Hamiltonian Floer theory]%
	{Computation of annular capacity by Hamiltonian Floer theory of non-contractible periodic trajectories}

\author{Morimichi Kawasaki} 
\address[Morimichi Kawasaki]{Center for Geometry and Physics, Institute for Basic Science (IBS), Pohang 790-784, Republic of Korea}
\email{kawasaki@ibs.re.kr}

\author{Ryuma Orita} 
\address[Ryuma Orita]{Graduate School of Mathematical Sciences, the University of Tokyo, Tokyo 153-0041, Japan}
\email{orita@ms.u-tokyo.ac.jp}
\urladdr{https://sites.google.com/site/oritaryuma/}


\newtheorem{theorem}{Theorem}[section]
\newtheorem{lemma}[theorem]{Lemma}
\newtheorem{proposition}[theorem]{Proposition}

\newtheorem{conjecture}[theorem]{Conjecture}
\newtheorem{step}{Step}
\newtheorem{claim}{Claim}

\theoremstyle{definition}
\newtheorem{definition}[theorem]{Definition}

\theoremstyle{remark}
\newtheorem{remark}[theorem]{Remark}

\newcommand{\grad}{\mathop{\mathrm{grad}}\nolimits}
\newcommand{\Ker}{\mathop{\mathrm{Ker}}\nolimits}
\newcommand{\Image}{\mathop{\mathrm{Im}}\nolimits}
\newcommand{\SH}{\mathop{\mathrm{SH}}\limits_{\longleftarrow}}
\newcommand{\RSH}{\mathop{\mathrm{SH}}\limits_{\longrightarrow}}

\newcommand{\relmiddle}[1]{\mathrel{}\middle#1\mathrel{}}

\begin{document}

\maketitle


\begin{abstract}
The first author \cite{Ka} introduced a relative symplectic capacity $C$ for a symplectic manifold $(N,\omega_N)$
and its subset $X$ which measures the existence of non-contractible periodic trajectories
of Hamiltonian isotopies on the product of $N$ with the annulus $A_R=(-R,R)\times\mathbb{R}/\mathbb{Z}$.
In the present paper, we give an exact computation of the capacity $C$ of the $2n$-torus $\mathbb{T}^{2n}$ relative to
a Lagrangian submanifold $\mathbb{T}^n$ which implies the existence of non-contractible Hamiltonian periodic trajectories
on $A_R\times\mathbb{T}^{2n}$.
Moreover, we give a lower bound on the number of such trajectories.
\end{abstract}

\tableofcontents

\section{Introduction}

Let $(N,\omega_N)$ be a closed connected symplectic manifold and $X\subset N$ a compact subset.
For $R>0$ let $A_R$ denote the annulus $(-R,R)\times\mathbb{R}/\mathbb{Z}$,
and for $u\in(-R,R)$ we define its subset $L_u=\{u\}\times\mathbb{R}/\mathbb{Z}\subset A_R$.
For $\ell\in\mathbb{Z}$ we put
\[
	\alpha_{\ell}=[t\mapsto (0,\ell t)]\in [S^1,A_R].
\]
We consider the product symplectic manifold $\bigl(A_R\times N,(dp_0\wedge dq_0)\oplus\omega_N\bigr)$
and the free homotopy class $(\alpha_{\ell},0_N)\in[S^1,A_R\times N]$,
where $(p_0,q_0)\in A_R=(-R,R)\times\mathbb{R}/\mathbb{Z}$ and $0_N\in[S^1,N]$ is the free homotopy class of trivial loops in $N$.
Our main result is Theorem \ref{theorem}.
We consider $(N,\omega_N)=(\mathbb{T}^{2n},\omega_{\mathrm{std}})=%
\bigl((\mathbb{R}/2\mathbb{Z}\times\mathbb{R}/\mathbb{Z})^n,\omega_{\mathrm{std}}\bigr)$
and $\mathbb{T}^n$ stands for $(\{0\}\times\mathbb{R}/\mathbb{Z})^n$
where $\omega_{\mathrm{std}}$ is the standard symplectic form on $\mathbb{T}^{2n}$.

\begin{theorem}\label{theorem}
Let $R>0$ and $u$ be real numbers such that $u\in(-R,R)$.
For any smooth Hamiltonian $H\colon [0,1]\times A_R\times \mathbb{T}^{2n}\to\mathbb{R}$ with compact support
and any $\ell\in\mathbb{Z}$ such that
\[
	\max\{R\lvert\ell\rvert+u\ell,0\}\leq c=\inf_{[0,1]\times L_u\times \mathbb{T}^n}H,
\]
there exists a Hamiltonian periodic trajectory $x$ in the homotopy class $(\alpha_{\ell},0_{\mathbb{T}^{2n}})\in[S^1,A_R\times\mathbb{T}^{2n}]$ with action $\mathcal{A}_H(x)\geq c-u\ell$,
where $\mathcal{A}_H$ is the action functional defined in Section \ref{section:2}.
Moreover, if $H$ is non-degenerate, then the number of such $x$'s is at least
\[
	b_{\mathbb{T}^{n+1}}=\sum_{k=0}^{n+1}b_k(\mathbb{T}^{n+1})=\sum_{k=0}^{n+1}\dim_{\mathbb{Z}/2\mathbb{Z}}\left(H_k(\mathbb{T}^{n+1};\mathbb{Z}/2\mathbb{Z})\right).
\]
\end{theorem}

The result in Theorem \ref{theorem} is sharp in the sense that for any $\epsilon>0$
there exists a Hamiltonian $H\colon [0,1]\times A_R\times \mathbb{T}^{2n}\to\mathbb{R}$
with $\inf_{[0,1]\times L_u\times \mathbb{T}^n}H=\max\{R\lvert\ell\rvert+u\ell,0\}-\epsilon$
without periodic trajectory in $(\alpha_{\ell},0_{\mathbb{T}^{2n}})$
(see the proof of Theorem \ref{theorem:main3} in Subsection \ref{subsection:proofofmain}).

To obtain Theorem \ref{theorem}, we will prove the non-zeroness of the homomorphism $T_{\alpha}^{[a,\infty);c}$ which is defined in Subsection \ref{subsection:symp}.
To prove the non-zeroness, we use Po\'{z}niak's theorem (Theorem \ref{theorem:Pozniak}) several times.
Biran, Polterovich and Salamon \cite{BPS}, Niche \cite{Ni} and Xue \cite{Xu} also used
Po\'{z}niak's theorem to get an upper bound of the Biran--Polterovich--Salamon capacity.
In their papers, the homomorphism $T_{\alpha}^{[a,\infty);c}$ is an isomorphism.
However in our case, $T_{\alpha}^{[a,\infty);c}$ is not an isomorphism.
Thus we need more sophisticated arguments.

We define a capacity $C$ introduced by the first author in \cite{Ka}
in terms of the Biran--Polterovich--Salamon capacity \cite[Subsection 3.2]{BPS} (see also \cite{Ni,We,Xu}).
Let $(M,\omega)$ be an open symplectic manifold and $A\subset M$ a compact subset.
For $\alpha\in [S^1,M]$ and $a\geq -\infty$ we define the \textit{Biran--Polterovich--Salamon capacity} $C_{\mathrm{BPS}}$ by
\[
	C_{\mathrm{BPS}}(M,A;\alpha,a)=%
	\inf\{\, c>0\mid \forall H\in\mathcal{H}_c(M,A)\ \exists x\in\mathcal{P}(H;\alpha)\ %
	\text{s.t.}\ \mathcal{A}_H(x)\geq a\,\}%
	\geq 0,
\]
where $\mathcal{H}_c(M,A)$ is the set of time-dependent Hamiltonian functions
$H\colon S^1\times M\to\mathbb{R}$ with compact support such that
\[
	\inf_{S^1\times A}H\geq c,
\]
and $\mathcal{P}(H;\alpha)$ is the set of periodic trajectories
of the Hamiltonian isotopy associated to $H$ representing $\alpha$.
For $R>0$, $u\in(-R,R)$, $\ell\in\mathbb{Z}$ and $a\geq -\infty$,
we then define a relative symplectic capacity $C(N,X;R,u,\ell,a)$ by
\[
	C(N,X;R,u,\ell,a)=C_{\mathrm{BPS}}(A_R\times N,L_u\times X;(\alpha_{\ell},0_N),a).
\]
By using the capacity $C$,
we can rewrite Theorem \ref{theorem} as Theorem \ref{theorem:main3} (see Subsection \ref{subsection:proofofmain} for details).

\begin{theorem}\label{theorem:main3}
For any $R>0$, $u\in(-R,R)$, $\ell\in\mathbb{Z}$ and $a\geq -\infty$, we have
\[
	C(\mathbb{T}^{2n},\mathbb{T}^n;R,u,\ell,a)=\max\{R\lvert\ell\rvert+u\ell,a+u\ell\}.
\]
\end{theorem}

After the first draft of this paper was completed, Ishiguro \cite{Is} pointed out the following proposition.

\begin{proposition}[{\cite[Proposition 5.1]{Is}}]
For any $R>0$, $u\in(-R,R)$, $\ell\in\mathbb{Z}$, $a\geq -\infty$ and $\beta\neq0\in[S^1,\mathbb{T}^{2n}]$, we have
\[
	C_{\mathrm{BPS}}(A_R\times \mathbb{T}^{2n},L_u\times \mathbb{T}^n;(\alpha_\ell,\beta),a)=\infty.
\]
\end{proposition}

On the other hand, the first author essentially showed Theorem \ref{theorem:kawasaki2} in \cite{Ka}.

\begin{theorem}[{\cite[Theorem 1.2]{Ka}}]\label{theorem:kawasaki2}
For any $R>0$ and $\ell\in\mathbb{Z}$ we have
\[
	C(\mathbb{T}^{2n},\mathbb{T}^n;R,0,\ell,-\infty)\leq 2R\lvert \ell\rvert.
\]
\end{theorem}

Therefore Theorem \ref{theorem:main3} improves Theorem \ref{theorem:kawasaki2}.
Moreover, the first author proposed the following conjecture.

\begin{conjecture}[{\cite[Conjecture 3.1]{Ka}}]\label{conjecture}
Let $X$ be a stably non-displaceable compact subset of a closed symplectic manifold $(M,\omega)$.
Show that the equality
\[
	C(M,X;R,0,\ell,-\infty)=R\lvert\ell\rvert
\]
holds for any $R>0$ and $\ell\in\mathbb{Z}$.
\end{conjecture}

Theorem \ref{theorem:main3} proves Conjecture \ref{conjecture} for $(M,X)=(\mathbb{T}^{2n},\mathbb{T}^n)$.

The paper is organized as follows.
In Section \ref{section:2}, we introduce the Floer homology and the symplectic homology for non-contractible trajectories which are the main tools to prove our main theorems.
In Section \ref{section:3}, we calculate the dimensions of the Floer homology and the symplectic homology to prove our main theorems.
In Section \ref{section:4}, we prove our main theorems (Theorems \ref{theorem} and \ref{theorem:main3}).


\section{Symplectic homology}\label{section:2}

In this section, we define the Floer homology for non-contractible periodic trajectories (see \cite[Section 4]{BPS} for details).

Let $(\overline{M},\omega)$ be a compact symplectic manifold with convex boundary $\partial\overline{M}$
(i.e., there exists a Liouville vector field defined on an open neighborhood of $\partial\overline{M}$ in $\overline{M}$ and pointing outward along $\partial\overline{M}$)
and denote $M=\overline{M}\setminus\partial\overline{M}$.
Although the product of compact symplectic manifolds with convex boundary need not have convex boundary,
we can still define the Floer homology of the product
according to \cite[\textsc{Products}, Section 3]{FS}.
In Section \ref{section:3}, we will consider the Floer homology of the product $\overline{A_R}\times\mathbb{T}^{2n}$
which has no convex boundary, where $\overline{A_R}=[-R,R]\times\mathbb{R}/\mathbb{Z}$.

\subsection{Action functional}

For a free homotopy class $\alpha\in [S^1,M]$
denote by $\mathcal{L}_{\alpha}M$ the space of free loops $S^1\to M$ representing $\alpha$.
In addition, we assume that our manifold $(\overline{M},\omega)$ is \textit{symplectically $\alpha$-atoroidal}, i.e.,
for any free loop $u$ in $\mathcal{L}_{\alpha}M$, that is for $u\colon S^1\to\mathcal{L}_{\alpha}M$
considered as the map $u\colon \mathbb{T}^2\to M$ from the two-torus,
\[
	\int_{\mathbb{T}^2} u^{\ast}\omega=0 \quad \text{and} \quad \int_{\mathbb{T}^2} u^{\ast}c_1=0
\]
hold where $c_1$ is the first Chern class.

Let $H\in\mathcal{H}=C_0^{\infty}(S^1\times M)$ be a Hamiltonian with compact support.
Let $H_t$ denote $H(t,\cdot)$ for $t\in S^1=\mathbb{R}/\mathbb{Z}$.
The \textit{Hamiltonian vector field} $X_H\in \mathfrak{X}(M)$ associated to $H$ is defined by
\[
	\iota_{X_H}\omega=-dH.
\]
The \textit{Hamiltonian isotopy} $\{\varphi_H^t\}_{t\in [0,1]}$ associated to $H$ is defined by
\[
	\begin{cases}
		\varphi_H^0=\mathrm{id},\\
		\frac{d}{dt}\varphi_H^t=X_{H_t}\circ\varphi_H^t\quad \text{for all}\ t\in [0,1],
	\end{cases}
\]
and its time-one map $\varphi_H=\varphi_H^1$ is referred to as the \textit{Hamiltonian diffeomorphism} of $H$.
Let $\mathcal{P}(H;\alpha)$ be the set of one-periodic trajectories of $\varphi_H$ representing $\alpha$.
A one-periodic trajectory $x\in\mathcal{P}(H;\alpha)$ is called \textit{non-degenerate}
if it satisfies $\det\bigl(d\varphi_H(x(0))-\mathrm{id}\bigr)\neq 0$.

Fix a reference loop $z\in\alpha$.
We define the \textit{action functional} $\mathcal{A}_H\colon \mathcal{L}_{\alpha}M\to \mathbb{R}$ by
\[
	\mathcal{A}_H(x)=-\int_{[0,1]\times S^1} \bar{x}^{\ast}\omega +\int_{0}^{1}H_t\bigl(x(t)\bigr)\, dt,
\]
where $H_t=H(t,\cdot)$ and
$\bar{x}$ is a path in $\mathcal{L}_{\alpha}M$ between $z$ and $x$
considered as a map $\bar{x}\colon [0,1]\times S^1\to M$ from the annulus $[0,1]\times S^1$ to $M$.
Since our manifold $(\overline{M},\omega)$ is symplectically $\alpha$-atoroidal,
the functional $\mathcal{A}_H$ is well-defined as a real-valued function.
Note that $\mathcal{P}(H;\alpha)$ is equal to the set of critical points of $\mathcal{A}_H$.

We define the \textit{action spectrum} of $\mathcal{A}_H$ by
\[
	\mathrm{Spec}(H;\alpha)=\mathcal{A}_H\bigl(\mathcal{P}(H;\alpha)\bigr).
\]
Let $a$ and $b$ be real numbers such that $-\infty\leq a<b\leq\infty$.
Suppose that the Hamiltonian $H$ satisfies $a,b\not\in\mathrm{Spec}(H;\alpha)$
and that it is \textit{regular}, i.e., every one-periodic trajectory $x\in\mathcal{P}(H;\alpha)$ is non-degenerate.
We define $\mathcal{P}^{[a,b)}(H;\alpha)=\mathcal{P}^b(H;\alpha)\setminus \mathcal{P}^a(H;\alpha)$
where $\mathcal{P}^a(H;\alpha)=\{\,x\in\mathcal{P}(H;\alpha)\mid \mathcal{A}_H(x)<a\,\}$.


\subsection{Filtered Floer chain complex}

We define the chain group of our Floer chain complex to be the $\mathbb{Z}/2\mathbb{Z}$-vector space
\[
	\mathrm{CF}^{[a,b)}(H;\alpha)=\mathrm{CF}^b(H;\alpha)/\mathrm{CF}^a(H;\alpha),
\]
where
\[
	\mathrm{CF}^a(H;\alpha)=\bigoplus_{x\in\mathcal{P}^{a}(H;\alpha)}\mathbb{Z}/2\mathbb{Z}\, x.
\]

Let $J_t=J_{t+1}\in\mathcal{J}(M,\omega)$ be a time-dependent smooth family of $\omega$-compatible almost complex structures on $\overline{M}$ such
that $J_t$ is convex and independent of $t$ near the boundary $\partial\overline{M}$.
Consider the Floer differential equation
\begin{equation}\label{eq:1}
	\partial_s u+J_t(u)\bigl(\partial_t u-X_{H_t}(u)\bigr)=0.
\end{equation}
Here we note that
\[
	\grad\mathcal{A}_H(u(s,\cdot))=J_t\bigl(u(s,\cdot)\bigr)\left(\partial_t u(s,\cdot)-X_{H_t}\bigl(u(s,\cdot)\bigr)\right)
\]
for all $s$.
For a smooth solution $u$ to \eqref{eq:1} we define the energy by the formula
\[
	E(u)=\int_0^1\int_{-\infty}^{\infty}\lvert\partial_s u\rvert^2\,dsdt.
\]
Then we have the following lemma.

\begin{lemma}[\cite{Sa}]
Let $u\colon \mathbb{R}\times S^1\to M$ be a smooth solution to \eqref{eq:1} with finite energy.
\begin{enumerate}
	\item There exist periodic solutions $x^{\pm}\in\mathcal{P}(H;\alpha)$ such that
		\[
			\lim_{s\to\pm\infty}u(s,t)=x^{\pm}(t)\quad \text{and}\quad \lim_{s\to\pm\infty}\partial_s u(s,t)=0,
		\]
		where both limits are uniform in the $t$-variable.
	\item The energy identity holds:
		\[
			E(u)=\mathcal{A}_H(x^-)-\mathcal{A}_H(x^+).
		\]
\end{enumerate}
\end{lemma}

We call a family of almost complex structures \textit{regular}
if the linearized operator for \eqref{eq:1} is surjective for any finite-energy solution of \eqref{eq:1}
in the homotopy class $\alpha$.
We denote by $\mathcal{J}_{\mathrm{reg}}(H;\alpha)$ the space of regular families of almost complex structures.
This subspace is generic in $\mathcal{J}(M,\omega)$ (see \cite{FHS}).
For any $J\in\mathcal{J}_{\mathrm{reg}}(H;\alpha)$ and any pair $x^{\pm}\in\mathcal{P}(H;\alpha)$
the space
\[
	\mathcal{M}(x^-,x^+;H,J)=\{\,\text{solution of \eqref{eq:1} satisfying (i)}\,\}
\]
is a smooth manifold whose dimension near such a solution $u$ is given by
the difference of the Conley--Zehnder indices (see \cite{SZ}) of $x^-$ and $x^+$ relative to $u$.
The subspace of solutions of relative index 1 is denoted by $\mathcal{M}^1(x^-,x^+;H,J)$.
For $J\in\mathcal{J}_{\mathrm{reg}}(H;\alpha)$ the quotient $\mathcal{M}^1(x^-,x^+;H,J)/\mathbb{R}$
is a finite set for any pair $x^{\pm}\in\mathcal{P}(H;\alpha)$.
We define the boundary operator $\partial^{H,J}\colon \mathrm{CF}^b(H;\alpha)\to \mathrm{CF}^b(H;\alpha)$ by
\[
	\partial^{H,J}(x)=\sum \#_2\left(\mathcal{M}^1(x,y;H,J)/\mathbb{R}\right)\,y
\]
for $x\in \mathcal{P}^b(H;\alpha)$ where $\#_2$ denotes the modulo 2 counting.

\begin{theorem}[\cite{Fl}]
If $J$ is regular, then the operator $\partial^{H,J}$ is well-defined and satisfies $\partial^{H,J}\circ\partial^{H,J}=0$.
\end{theorem}

The energy identity (ii) implies that $\mathrm{CF}^a(H;\alpha)$ is invariant under the boundary operator $\partial^{H,J}$.
Thus we get an operator $[\partial^{H,J}]$ on the quotient $\mathrm{CF}^{[a,b)}(H;\alpha)$.

\begin{definition}
The \textit{filtered Floer homology group} is defined to be
\[
	\mathrm{HF}^{[a,b)}(H,J;\alpha)=\Ker{[\partial^{H,J}]}/\Image{[\partial^{H,J}]}.
\]
\end{definition}

\begin{theorem}[\cite{Fl,Sa,SZ}]
If $J_0, J_1\in\mathcal{J}(H;\alpha)$ are two regular almost complex structures, then there exists a natural isomorphism
\[
	\mathrm{HF}^{[a,b)}(H,J_0;\alpha)\to \mathrm{HF}^{[a,b)}(H,J_1;\alpha).
\]
\end{theorem}

We refer to $\mathrm{HF}^{[a,b)}(H;\alpha)=\mathrm{HF}^{[a,b)}(H,J;\alpha)$ as the Floer homology associated to $H$.


\subsection{Continuation}

We define the set
\[
	\mathcal{H}^{a,b}(M;\alpha)=\{\,H\in\mathcal{H}\mid a,b\not\in\mathrm{Spec}(H;\alpha)\,\}.
\]

\begin{proposition}[{\cite[Remark 4.4.1]{BPS}}]\label{proposition:nbd}
Every Hamiltonian $H\in\mathcal{H}^{a,b}(M;\alpha)$ has a neighborhood $\mathcal{U}$
such that the Floer homology groups $\mathrm{HF}^{[a,b)}(H',J';\alpha)$, for any regular $H'\in \mathcal{U}$
and any regular almost complex structure $J'\in\mathcal{J}_{\mathrm{reg}}(H';\alpha)$, are naturally isomorphic.
\end{proposition}

According to Proposition \ref{proposition:nbd}, one can define the Floer homology $\mathrm{HF}^{[a,b)}(H;\alpha)$
whether $H$ is regular or not.

\begin{definition}\label{definition:degenerate}
For $H\in\mathcal{H}^{a,b}(M;\alpha)$ we define $\mathrm{HF}^{[a,b)}(H;\alpha)=\mathrm{HF}^{[a,b)}(\widetilde{H};\alpha)$,
where $\widetilde{H}$ is any regular Hamiltonian sufficiently close to $H$.
\end{definition}

\begin{remark}[{\cite[Remark 4.4.2]{BPS}}]\label{remark:0}
We can define the filtered Floer homology $\mathrm{HF}^{[a,b)}(H;0_M)$ if the action interval $[a,b)$ does not contain zero,
i.e, either $-\infty\leq a<b<0$ or $0<a<b\leq\infty$.
\end{remark}


\subsection{Monotone homotopies}\label{subsection:monotone}
We introduce a bidirected partial order on $\mathcal{H}=C_0^{\infty}(S^1\times M)$ by
\[
	H_0\preceq H_1\quad\iff\quad H_0(t,x)\geq H_1(t,x)\qquad\text{for all}\ (t,x)\in S^1\times M.
\]
Then there exists a homotopy $\{H_s\}_s$ from $H_0$ to $H_1$ such that $\partial_s H_s\leq 0$.
We call such a homotopy of Hamiltonians \textit{monotone}.
Let $\alpha\in [S^1,M]$ be a nontrivial free homotopy class and $a,b\in\mathbb{R}\cup\{\infty\}$ such that $a<b$.
It follows from the energy identity
\[
	E(u)=\mathcal{A}_{H_0}(x^-)-\mathcal{A}_{H_1}(x^+)+\int_0^1\int_{-\infty}^{\infty}\partial_s H\bigl(s,t,u(s,t)\bigr)\,dsdt
\]
that the Floer chain map $\Phi_{H_1H_0}\colon\mathrm{CF}(H_0;\alpha)\to\mathrm{CF}(H_1;\alpha)$,
defined in terms of the solutions of the equation
\[
	\partial_s u+J_{s,t}(u)\bigl(\partial_t u-X_{H_{s,t}}(u)\bigr)=0,
\]
preserves the subcomplexes $\mathrm{CF}^a(H_0;\alpha)$ and $\mathrm{CF}^b(H_0;\alpha)$.
Hence every monotone homotopy $\{H_s\}_s$ induces a natural homomorphism
\[
	\sigma_{H_1H_0}\colon \mathrm{HF}^{[a,b)}(H_0;\alpha)\to \mathrm{HF}^{[a,b)}(H_1;\alpha)
\]
whenever $H_0,H_1\in\mathcal{H}^{a,b}(M;\alpha)$ satisfy $H_0\preceq H_1$ (see \cite[Subsection 4.5]{BPS}).
The homomorphism $\sigma_{H_1H_0}$ is called the \textit{monotone homomorphism} from $H_0$ to $H_1$.
We call a monotone homotopy $\{H_s\}_s$ \textit{action-regular}
if $H_s$ takes values in a connected component of $\mathcal{H}^{a,b}(M;\alpha)$.

\begin{lemma}[\cite{FH,CFH}]\label{lemma:functoriality}
The monotone homomorphism is independent of the choice of the monotone homotopy used to define it and
\[
	\sigma_{H_2H_1}\circ \sigma_{H_1H_0}=\sigma_{H_2H_0},\quad \sigma_{H_0H_0}=\mathrm{id},
\]
whenever $H_0,H_1,H_2\in\mathcal{H}^{a,b}(M;\alpha)$ satisfy $H_0\preceq H_1\preceq H_2$.
\end{lemma}

\begin{lemma}[\rm \cite{Vi}]\label{lemma:actionregular}
The monotone homomorphism associated to an action-regular monotone homotopy is an isomorphism.
\end{lemma}

Given $-\infty\leq a<b<c\leq\infty$ and
two Hamiltonians $H_0,H_1\in\mathcal{H}^{a,b}(M;\alpha)\cap\mathcal{H}^{a,c}(M;\alpha)$
satisfying $H_0\preceq H_1$, we obtain the following commutative diagram,
whose rows are the short exact sequences for $H_0$ and for $H_1$.
\[
\xymatrix{
0 \ar[r] & \mathrm{CF}_{\ast}^{[a,b)}(H_0;\alpha) \ar[d]^{\Phi_{H_1H_0}} \ar[r]^{\iota^F}  & \mathrm{CF}_{\ast}^{[a,c)}(H_0;\alpha) \ar[d]^{\Phi_{H_1H_0}} \ar[r]^{\pi^F}  & \mathrm{CF}_{\ast}^{[b,c)}(H_0;\alpha) \ar[r] \ar[d]^{\Phi_{H_1H_0}} & 0 \\
0 \ar[r] & \mathrm{CF}_{\ast}^{[a,b)}(H_1;\alpha) \ar[r]^{\iota^F} & \mathrm{CF}_{\ast}^{[a,c)}(H_1;\alpha) \ar[r]^{\pi^F} & \mathrm{CF}_{\ast}^{[b,c)}(H_1;\alpha) \ar[r] & 0 \\
}
\]
where $\iota^F$ and $\pi^F$ denote the natural inclusion and projection, respectively.
The associated long exact sequences induce the following commutative diagram.
\begin{equation}\label{eq:longexactseq}
\xymatrix{
\cdots \ar[r] & \mathrm{HF}_{\ast}^{[a,b)}(H_0;\alpha) \ar[d]^{\sigma_{H_1H_0}} \ar[r]^{[\iota^F]}  & \mathrm{HF}_{\ast}^{[a,c)}(H_0;\alpha) \ar[d]^{\sigma_{H_1H_0}} \ar[r]^{[\pi^F]}  & \mathrm{HF}_{\ast}^{[b,c)}(H_0;\alpha) \ar[r] \ar[d]^{\sigma_{H_1H_0}} & \cdots \\
\cdots \ar[r] & \mathrm{HF}_{\ast}^{[a,b)}(H_1;\alpha) \ar[r]^{[\iota^F]} & \mathrm{HF}_{\ast}^{[a,c)}(H_1;\alpha) \ar[r]^{[\pi^F]} & \mathrm{HF}_{\ast}^{[b,c)}(H_1;\alpha) \ar[r] & \cdots \\
}
\end{equation}


\subsection{Symplectic homology}\label{subsection:symp}

In this subsection, we consider a homology introduced in \cite{FH,CFH,Ci}. We refer to \cite[Subsection 4.8]{BPS} for details.
Let $\alpha\in [S^1,M]$ be a nontrivial free homotopy class and $a,b\in\mathbb{R}\cup\{\infty\}$ such that $a<b$.
As mentioned in Subsection \ref{subsection:monotone}, there is a natural homomorphism
\[
	\sigma_{H_1H_0}\colon \mathrm{HF}^{[a,b)}(H_0;\alpha)\to \mathrm{HF}^{[a,b)}(H_1;\alpha)
\]
whenever $H_0,H_1\in\mathcal{H}^{a,b}(M;\alpha)$ satisfy $H_0\preceq H_1$.
These homomorphisms define an inverse system of Floer homology groups over $\bigl(\mathcal{H}^{a,b}(M;\alpha),\preceq\bigr)$.
We denote the \textit{symplectic homology} of $M$ in the homotopy class $\alpha$ for the action interval $[a,b)$ by
\[
	{\SH}^{[a,b)}(M;\alpha)=%
	\lim_{\substack{\longleftarrow\\ H\in\mathcal{H}^{a,b}(M;\alpha)}}\mathrm{HF}^{[a,b)}(H;\alpha).
\]
Fix a compact subset $A\subset M$ and a constant $c\in\mathbb{R}$.
We define the set
\[
	\mathcal{H}_c^{a,b}(M,A;\alpha)=%
	\left\{\,H\in\mathcal{H}^{a,b}(M;\alpha)\relmiddle| \inf_{S^1\times A}H>c\,\right\}.
\]
This defines a directed system of Floer homology groups over $\bigl(\mathcal{H}_c^{a,b}(M,A;\alpha),\preceq\bigr)$.
We denote the \textit{relative symplectic homology} of the pair $(M,A)$ at the \textit{level} $c$
in the homotopy class $\alpha$ for the action interval $[a,b)$ by
\[
	{\RSH}^{[a,b);c}(M,A;\alpha)=%
	\lim_{\substack{\longrightarrow\\ H\in\mathcal{H}_c^{a,b}(M,A;\alpha)}}\mathrm{HF}^{[a,b)}(H;\alpha).
\]

\begin{proposition}[{\cite[Proposition 4.8.2]{BPS}}]\label{proposition:diagram}
Let $\alpha\in [S^1,M]$ be a nontrivial homotopy class and suppose that $-\infty\leq a<b\leq\infty$.
Then for any $c\in\mathbb{R}$ there exists a unique homomorphism
\[
	T_{\alpha}^{[a,b);c}\colon {\SH}^{[a,b)}(M;\alpha)\to {\RSH}^{[a,b);c}(M,A;\alpha)
\]
such that for any two Hamiltonians $H_0,H_1\in\mathcal{H}_c^{a,b}(M,A;\alpha)$ with $H_0\geq H_1$
the following diagram commutes.
\[
	\xymatrix{
	{\SH}^{[a,b)}(M;\alpha) \ar[d]_{\pi_{H_0}} \ar[rr]^{T_{\alpha}^{[a,b);c}} & & {\RSH}^{[a,b);c}(M,A;\alpha) \\
	\mathrm{HF}^{[a,b)}(H_0;\alpha) \ar[rr]^{\sigma_{H_1H_0}} & & \mathrm{HF}^{[a,b)}(H_1;\alpha) \ar[u]_{\iota_{H_1}} \\
	}
\]
Here $\pi_{H_0}$ and $\iota_{H_1}$ are the canonical homomorphisms.
In particular, since $\sigma_{HH}=\mathrm{id}$ for any $H\in\mathcal{H}_c^{a,b}(M,A;\alpha)$, the following diagram commutes.
\[
	\xymatrix{
	{\SH}^{[a,b)}(M;\alpha) \ar[rd]_{\pi_H} \ar[rr]^{T_{\alpha}^{[a,b);c}} & & {\RSH}^{[a,b);c}(M,A;\alpha) \\
	& \mathrm{HF}^{[a,b)}(H;\alpha) \ar[ru]_{\iota_H} & \\
	}
\]
\end{proposition}

\begin{remark}
As we noted in Remark \ref{remark:0}, we can still define the filtered symplectic homology for $\alpha=0_M$
and the conclusion of Proposition \ref{proposition:diagram} still holds if the action interval $[a,b)$ does not contain zero.
\end{remark}


\section{Computation}\label{section:3}

\subsection{Morse--Bott theory in Floer homology}

We refer to \cite[Subsection 5.2]{BPS} for details.

\begin{definition}[\cite{BPS}]
A subset $P\subset\mathcal{P}(H)$ is called a \textit{Morse--Bott manifold of periodic trajectories for} $H$
if the set $C_0=\{\,x(0)\mid x\in P\,\}$ is a compact submanifold of $M$ and
$T_{x_0}C_0=\Ker\bigl(d\varphi_H^1(x(0))-\mathrm{id}\bigr)$ for any $x_0\in C_0$.
\end{definition}

\begin{theorem}[{\cite[Theorem 5.2.2]{BPS}}]\label{theorem:Pozniak}
Let $-\infty\leq a<b\leq \infty$, $\alpha\in [S^1,M]$, and $H\in\mathcal{H}^{a,b}(M;\alpha)$.
Assume that the set $P=\{\,x\in\mathcal{P}(H;\alpha)\mid a<\mathcal{A}_H(x)<b\,\}$
is a connected Morse--Bott manifold of periodic trajectories for $H$.
Let $g$ be a Riemannian metric on $C_0$ and $f\colon C_0\to\mathbb{R}$ a Morse--Smale function.
Then the Floer homology $\mathrm{HF}_{\ast}^{[a,b)}(H;\alpha)=H_{\ast}\bigl(\mathrm{CF}_{\ast}^{[a,b)}(H;\alpha),\partial_{\ast}\bigr)$
coincides with the Morse homology $\mathrm{HM}_{\ast}(C_0,f,g)=H_{\ast}\bigl(\mathrm{CM}_{\ast}(C_0,f,g),\partial_{\ast}\bigr)$.
Namely, we have
\[
	\mathrm{HF}_{\ast}^{[a,b)}(H;\alpha)%
	\cong\mathrm{HM}_{\dim{C_0}-\ast}(C_0,f,g)%
	\cong H_{\dim{P}-\ast}(P;\mathbb{Z}/2\mathbb{Z}).
\]
\end{theorem}

The original version of Theorem \ref{theorem:Pozniak} can be found in \cite[Theorem 3.4.11]{Po}.

\begin{remark}
The grading of the Floer homology groups is well-defined up to an additive constant.
More precisely, with a suitable choice of this grading,
$\mathrm{HF}_k^{[a,b)}(H;\alpha)$ is isomorphic to $\mathrm{HM}_{\dim{C_0}-k}(C_0,f,g)$,
and hence to $H_k(P;\mathbb{Z}/2\mathbb{Z})$ if $P$ is orientable.
In the present paper, we choose this grading for simplicity.
\end{remark}


\subsection{Dimensions of symplectic homology groups}

Let $(N,\omega_N)$ be a closed connected symplectic manifold and $X\subset N$ a compact subset.
Let $A_R$ denote the annulus $(-R,R)\times\mathbb{R}/\mathbb{Z}$,
and for $u\in(-R,R)$ we define its subset $L_u=\{u\}\times\mathbb{R}/\mathbb{Z}\subset A_R$.
For $\ell\in\mathbb{Z}$ we put
\[
	\alpha_{\ell}=[t\mapsto (0,\ell t)]\in [S^1,A_R].
\]
We consider the product symplectic manifold $\bigl(A_R\times N,(dp_0\wedge dq_0)\oplus\omega_N\bigr)$
and the free homotopy class $(\alpha_{\ell},0_N)\in[S^1,A_R\times N]$,
where $(p_0,q_0)\in A_R=(-R,R)\times\mathbb{R}/\mathbb{Z}$.

In this subsection, we give an explicit computation of symplectic homology groups in the case that
$(N,\omega_N)=(\mathbb{T}^{2n},\omega_{\mathrm{std}})=%
\bigl((\mathbb{R}/2\mathbb{Z}\times\mathbb{R}/\mathbb{Z})^n,\omega_{\mathrm{std}}\bigr)$
and $X=(\{0\}\times\mathbb{R}/\mathbb{Z})^n\cong \mathbb{T}^n$.
In the following theorem, let $0$ denote the trivial homotopy class $(\alpha_0,0_{\mathbb{T}^{2n}})\in[S^1,A_R\times\mathbb{T}^{2n}]$.

\begin{theorem}\label{theorem:main0}
Let $R>0$ and $u$ be real numbers such that $u\in(-R,R)$.
Then for any $a>0$ and any $c>0$, we have
\[
	{\SH}^{[a,\infty)}(A_R\times \mathbb{T}^{2n};0)\cong
	H_{\ast}(\mathbb{T}^{2n+1};\mathbb{Z}/2\mathbb{Z}),
\]
and
\[
	{\RSH}^{[a,\infty);c}(A_R\times \mathbb{T}^{2n},L_u\times \mathbb{T}^n;0)=%
	\begin{cases}
		H_{\ast}(\mathbb{T}^{n+1};\mathbb{Z}/2\mathbb{Z}) & \text{if\: $0<a\leq c$}, \\
		0 & \text{if\: $a>c$}.
	\end{cases}
\]
\end{theorem}

For $\ell\neq 0$, we have the following result.
For the sake of brevity, we put
\[
	\hat{\alpha}_{\ell}=(\alpha_{\ell},0_{\mathbb{T}^{2n}})\in[S^1,A_R\times\mathbb{T}^{2n}].
\]

\begin{theorem}\label{theorem:main}
Let $R>0$ and $u$ be real numbers such that $u\in(-R,R)$, and $\ell\in\mathbb{Z}\setminus\{0\}$.
Then for any $a\in\mathbb{R}$ and $c>\max\{{u\ell,0\}}$, we have
\[
	{\SH}^{[a,\infty)}(A_R\times \mathbb{T}^{2n};\hat{\alpha}_{\ell})=%
	\begin{cases}
		0 & \text{if\: $a<R\lvert\ell\rvert$}, \\
		H_{\ast}(\mathbb{T}^{2n+1};\mathbb{Z}/2\mathbb{Z}) & \text{if\: $a\geq R\lvert\ell\rvert$},
	\end{cases}
\]
and
\[
	{\RSH}^{[a,\infty);c}(A_R\times \mathbb{T}^{2n},L_u\times \mathbb{T}^n;\hat{\alpha}_{\ell})=%
	\begin{cases}
		H_{\ast}(\mathbb{T}^{n+1};\mathbb{Z}/2\mathbb{Z}) & \text{if\: $0<a\leq c-u\ell$}, \\
		0 & \text{if\: $a>c-u\ell$}.
	\end{cases}
\]
\end{theorem}

\begin{theorem}\label{theorem:main2}
Let $R>0$ and $u$ be real numbers such that $u\in(-R,R)$ and let $\ell\in\mathbb{Z}$.
The homomorphism
\[
	\left(T_{\hat{\alpha}_{\ell}}^{[a,\infty);c}\right)_k\colon {\SH}{}_k^{[a,\infty)}(A_R\times \mathbb{T}^{2n};\hat{\alpha}_{\ell})%
	\to {\RSH}{}_k^{[a,\infty);c}(A_R\times \mathbb{T}^{2n},L_u\times \mathbb{T}^n;\hat{\alpha}_{\ell})
\]
is non-zero if and only if $R\lvert\ell\rvert<a\leq c-u\ell$ and $k=0,1,\ldots,n+1$.
Moreover, in this case, the homomorphism $\left(T_{\hat{\alpha}_{\ell}}^{[a,\infty);c}\right)_k$ is surjective and
\[
	\dim_{\mathbb{Z}/2\mathbb{Z}}\left(\Image\left(T_{\hat{\alpha}_{\ell}}^{[a,\infty);c}\right)_k\right)=%
	b_k(\mathbb{T}^{n+1}),
\]
where $b_k(\mathbb{T}^{n+1})$ is the $k$-th Betti number of $\mathbb{T}^{n+1}$.
\end{theorem}

We shall denote a point in $M=A_R\times \mathbb{T}^{2n}$ by
\[
	x=(p,q)=(p_0,q_0;p_1,q_1,\ldots,p_n,q_n),
\]
where $(p_0,q_0)\in A_R=(-R,R)\times\mathbb{R}/\mathbb{Z}$
and $(p_1,q_1,\ldots,p_n,q_n)\in \mathbb{T}^{2n}=(\mathbb{R}/2\mathbb{Z}\times\mathbb{R}/\mathbb{Z})^n$.
Then we have
\[
	L_u\times \mathbb{T}^n=\{\,p_0=u\ \text{and}\ p_1=\cdots=p_n=0\,\}.
\]
We put $m_u=\min\{1,R-\lvert u\rvert\}$
and $\overline{u}=(\overline{u}_0,\overline{u}_1,\ldots,\overline{u}_n)=(u,0,\ldots,0)\in (-R,R)\times (\mathbb{R}/2\mathbb{Z})^n$.


\subsection{Proof of Theorem \ref{theorem:main0}}

\begin{proof}
The proof of \cite[Theorem 5.1.1]{BPS} carries over almost literally.
Fix a positive real number $c>0$
and choose a smooth family of real functions $\{f_s(r)\}_{s\in\mathbb{R}}$,
defined for $r\in\mathbb{R}$, with the following properties (cf.\ \cite[Subsection 5.4]{BPS}).
\begin{enumerate}
	\item $f_s(-r)=f_s(r)$ for all $s$ and $r$.
	\item For any $s\in\mathbb{R}$
	\[
		f_s(0)>c,\quad f_s''(0)<0,\quad \lim_{s\to -\infty}f_s(0)=c,\quad \lim_{s\to\infty}f_s(0)=\infty.
	\]
	\item For all $s$ and $r$ we have $\partial_s f_s(r)\geq 0$.
	\item If $s\geq 1$, then
	\[
		f_s(r) =%
		\begin{cases}
			f_s(0)(1-r^2) & \text{if\: $0\leq r\leq 1-\frac{1}{4s}$}, \\
			0 & \text{if\: $r\geq 1-\frac{1}{8s}$},
		\end{cases}
	\]
	and $f'_s(r)\leq 0$ for all $r\geq 0$.
	\item If $s\leq -1$, then
	\[
		f_s(r) =%
		\begin{cases}
			f_s(0)(1-r^2) & \text{if\: $0\leq r\leq\frac{1}{8\lvert s\rvert}$}, \\
			s & \text{if\: $\frac{1}{4\lvert s\rvert}\leq r\leq 1-\frac{1}{4\lvert s\rvert}$}, \\
			0 & \text{if\: $r\geq 1-\frac{1}{8\lvert s\rvert}$},
		\end{cases}
	\]
	$f'_s(r)\leq 0$ for $r\leq 1/2$ and $f'_s(r)\geq 0$ for $r\geq 1/2$.
	\item For any $s$ the only critical point $r$ of $f_s$ with $f_s(r)>0$ is $r=0$.
\end{enumerate}

Now choose a family of Hamiltonians $H_s\colon A_R\times \mathbb{T}^{2n}\to\mathbb{R}$ with compact support so that
for $s\geq 1$
\[
	H_s(p,q)=f_s\left(\frac{p_0}{R}\right),
\]
and for $s\leq -1$
\[
	H_s(p,q) =%
	\begin{cases}
		f_s\left(\frac{\left\| p-\overline{u}\right\|}{m_u}\right) & \text{if\: $\left\| p-\overline{u}\right\|\leq \frac{m_u}{2}$}, \\
		f_s\left(\frac{p_0-\lvert u\rvert}{R-\lvert u\rvert}\right) & \text{if\: $p_0\geq \frac{R+\lvert u\rvert}{2}$}, \\
		f_s\left(\frac{p_0+\lvert u\rvert}{R-\lvert u\rvert}\right) & \text{if\: $p_0\leq -\frac{R+\lvert u\rvert}{2}$}, \\
		s & \text{otherwise}.
	\end{cases}
\]
We note that every contractible trajectory in $A_R\times\mathbb{T}^{2n}$ is constant.
By (vi), for $s\geq 1$ the set $\mathcal{P}(H_s;0)$ of contractible periodic trajectories is denoted by
\[
	\mathcal{P}(H_s;0)=\left\{\,t\mapsto \bigl(p(t),q(t)\bigr)\relmiddle|%
	\begin{aligned}%
	\dot{q}_0\equiv\dot{q}_1\equiv\cdots\equiv \dot{q}_n\equiv 0,\\ %
	p_0\equiv 0,\ \dot{p}_1\equiv\cdots\equiv\dot{p}_n\equiv 0
	\end{aligned}\,\right\},
\]
and for $s\leq -1$
\[
	\mathcal{P}(H_s;0)=\left\{\,t\mapsto \bigl(p(t),q(t)\bigr)\relmiddle|%
	\begin{aligned}%
	\dot{q}_0\equiv\dot{q}_1\equiv\cdots\equiv \dot{q}_n\equiv 0,\\ %
	p_0\equiv u,\ p_1\equiv\cdots\equiv p_n\equiv 0
	\end{aligned}\,\right\}.
\]

\begin{lemma}[{\cite[Lemma 5.3.1]{BPS}}]\label{lemma:MB0}
For every $s\in\mathbb{R}$ the set $\mathcal{P}(H_s;0)$ is
a Morse--Bott manifold of periodic trajectories for $H_s$ if and only if $f''_s(0)\neq 0$.
\end{lemma}

By (ii) and Lemma \ref{lemma:MB0},
the set $\mathcal{P}(H_s;0)$ is a Morse--Bott manifold of periodic trajectories for $H_s$.
Moreover, $\mathcal{P}(H_s;0)\cong\mathbb{T}^{2n+1}$ for $s\geq 1$
and $\mathcal{P}(H_s;0)\cong\mathbb{T}^{n+1}$ $s\leq -1$.
For every $s\in\mathbb{R}$, the action of $x\in\mathcal{P}(H_s;0)$ is
\[
	\mathcal{A}_H(x)=f_s(0)>c.
\]
By Theorem \ref{theorem:Pozniak}, for every $s\in\mathbb{R}$ we have
\[
	\mathrm{HF}_{\ast}^{[a,\infty)}(H_s;0)\cong%
	\begin{cases}%
		H_{\ast}\bigl(\mathcal{P}(H_s;0);\mathbb{Z}/2\mathbb{Z}\bigr) & \text{if\ $0<a<f_s(0)$},\\
		0 & \text{if\ $0<f_s(0)<a$}.
	\end{cases}
\]
By \cite[Proposition 4.5.1]{BPS}, the monotone homomorphism
\[
	\sigma_{H_{s_1}H_{s_0}}\colon \mathrm{HF}^{[a,\infty)}(H_{s_0};0)\to\mathrm{HF}^{[a,\infty)}(H_{s_1};0)
\]
is an isomorphism whenever $a\not\in[f_{s_1}(0),f_{s_0}(0)]$
and either $1\leq s_1\leq s_0$ or $s_1\leq s_0\leq -1$.
By \cite[Lemma 4.7.1 (ii)]{BPS}, the homomorphism
\[
	\pi_{H_s}\colon {\SH}^{[a,\infty)}(A_R\times \mathbb{T}^{2n};0)\to \mathrm{HF}^{[a,\infty)}(H_s;0)
\]
is an isomorphism for any $s\geq 1$ such that $f_s(0)>a$.
Therefore,
\[
	{\SH}^{[a,\infty)}(A_R\times \mathbb{T}^{2n};0)\cong%
	H_{\ast}(\mathbb{T}^{2n+1};\mathbb{Z}/2\mathbb{Z}).
\]
By \cite[Lemma 4.7.1 (i)]{BPS}, the homomorphism
\[
	\iota_{H_s}\colon \mathrm{HF}^{[a,\infty)}(H_s;0)\to {\RSH}^{[a,\infty);c}(A_R\times \mathbb{T}^{2n},L_u\times\mathbb{T}^n;0)
\]
is an isomorphism for any $s\leq -1$ if $a\leq c$,
and for any $s\leq -1$ with $f_s(0)<a$ if $a>c$.
Hence we obtain
\[
	{\RSH}^{[a,\infty);c}(A_R\times \mathbb{T}^{2n},L_u\times \mathbb{T}^n;0)=%
	\begin{cases}
		H_{\ast}(\mathbb{T}^{n+1};\mathbb{Z}/2\mathbb{Z}) & \text{if\: $0<a\leq c$}, \\
		0 & \text{if\: $a>c$}.
	\end{cases}
\]
Thus the proof of Theorem \ref{theorem:main0} is complete.
\end{proof}


\subsection{Proof of Theorem \ref{theorem:main}}

\begin{proof}
We assume, without loss of generality, that $\ell>0$.
Fix a positive real number $c>\max\{u\ell,0\}$ and choose a smooth family of real functions $\{f_s(r)\}_{s\in\mathbb{R}}$,
defined for $r\in\mathbb{R}$, with the following properties (cf.\ \cite[Subsection 5.5]{BPS}).
\begin{enumerate}
	\item $f_s(-r)=f_s(r)$ for all $s$ and $r$.
	\item For any $s\in\mathbb{R}$
	\[
		f_s(0)>c,\quad \lim_{s\to -\infty}f_s(0)=c,\quad \lim_{s\to\infty}f_s(0)=\infty.
	\]
	\item For all $s$ and $r$ we have $\partial_s f_s(r)\geq 0$.
	\item If $s\geq 1$, then
	\[
		f_s(r) =%
		\begin{cases}
			f_s(0) & \text{if\: $0\leq r\leq 1-\frac{3}{8s}$}, \\
			0 & \text{if\: $r\geq 1-\frac{1}{8s}$},
		\end{cases}
	\]
	$f'_s(r)\leq 0$ for all $r\geq 0$, and
	\[
		\begin{cases}
			f''_s(r)<0 & \text{if\: $1-\frac{3}{8s}<r<1-\frac{2}{8s}$}, \\
			f''_s(r)>0 & \text{if\: $1-\frac{2}{8s}<r<1-\frac{1}{8s}$}.
		\end{cases}
	\]
	\item If $s\leq -1$, then
	\[
		f_s(r) =%
		\begin{cases}
			f_s(0) & \text{if\: $0\leq r\leq\frac{1}{8\lvert s\rvert}$}, \\
			s & \text{if\: $\frac{3}{8\lvert s\rvert}\leq r\leq 1-\frac{3}{8\lvert s\rvert}$}, \\
			0 & \text{if\: $r\geq 1-\frac{1}{8\lvert s\rvert}$},
		\end{cases}
	\]
	$f'_s(r)\leq 0$ for $r\leq 1/2$, $f'_s(r)\geq 0$ for $r\geq 1/2$, and
	\[
		\begin{cases}
			f''_s(r)<0 & \text{if\: $\frac{1}{8\lvert s\rvert}<r<\frac{2}{8\lvert s\rvert}$}, \\
			f''_s(r)>0 & \text{if\: $\frac{2}{8\lvert s\rvert}<r<\frac{3}{8\lvert s\rvert}$}.
		\end{cases}
	\]
	\item For any $s\geq 1$ such that $f_s(0)>R\ell$ there exist real numbers $r'_s<r_s<0$ such that
	\[
		f'_s(r_s)=f'_s(r'_s)=R\ell,\quad f''_s(r_s)<0,\quad f''_s(r'_s)>0,
	\]
	and $f'_s(r)\neq R\ell$ for any $r\in \mathbb{R}\setminus \{r_s,r'_s\}$.
	\item For any $s\leq -1$
	there exist real numbers $r'_s<r_s<0$ such that
	\[
		f'_s(r_s)=f'_s(r'_s)=m_u\ell,\quad f''_s(r_s)<0,\quad f''_s(r'_s)>0,
	\]
	and $f'_s(r)\neq m_u\ell$ for any $r\in \mathbb{R}\setminus \{r_s,r'_s\}$.
	\item For any $s\leq -1$ the only possible points $r>0$ with $f'_s(r)=(R-\lvert u\rvert)\ell$ must satisfy $f_s(r)<0$
	and the number of such $r$ is finite.
\end{enumerate}
\begin{center}
	\begin{overpic}[width=8truecm,clip]{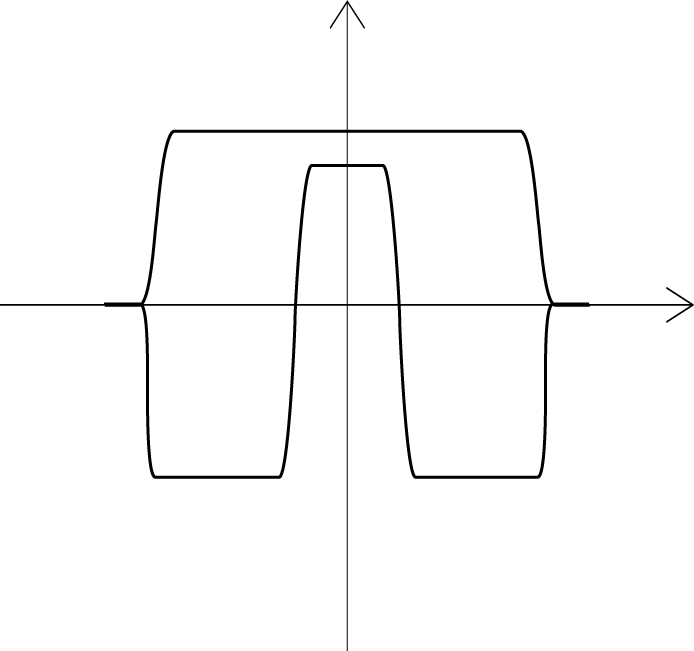}
	\put(115,155){\line(-1,0){2}}
	\put(116,104){$0$}
	\put(116,150){$c$}
	\put(191,104){$1$}
	\put(25,104){$-1$}
	\put(230,111){$r$}
	\put(180,170){$s\geq 1$}
	\put(16,46){$s\leq -1$}
	\end{overpic}
\captionof{figure}{Outline of the graphs of $f_s$ (for $s\geq 1$ and $s\leq -1$).}\label{figure:1}
\end{center}

Now choose a family of Hamiltonians $H_s\colon A_R\times \mathbb{T}^{2n}\to\mathbb{R}$ with compact support so that
for $s\geq 1$
\[
	H_s(p,q)=f_s\left(\frac{p_0}{R}\right),
\]
and for $s\leq -1$
\[
	H_s(p,q) =%
	\begin{cases}
		f_s\left(\frac{\left\| p-\overline{u}\right\|}{m_u}\right) & \text{if\: $\left\| p-\overline{u}\right\|\leq \frac{m_u}{2}$}, \\
		f_s\left(\frac{p_0-\lvert u\rvert}{R-\lvert u\rvert}\right) & \text{if\: $p_0\geq \frac{R+\lvert u\rvert}{2}$}, \\
		f_s\left(\frac{p_0+\lvert u\rvert}{R-\lvert u\rvert}\right) & \text{if\: $p_0\leq -\frac{R+\lvert u\rvert}{2}$}, \\
		s & \text{otherwise}.
	\end{cases}
\]
Here we note that the condition $m_u\leq 1$ ensures that the Hamiltonians $H_s$, $s\leq -1$ are well-defined.
\begin{center}
	\begin{overpic}[width=6truecm,clip]{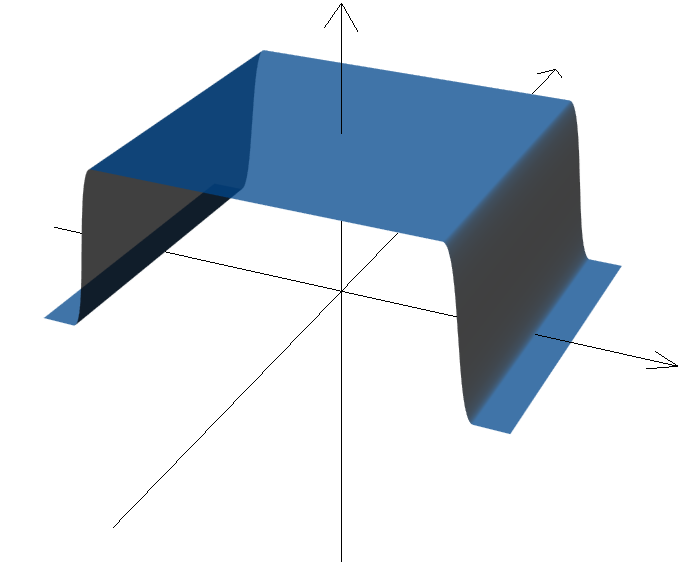}
	\put(88,63){$0$}
	\put(162,45){$p_0$}
	\put(140,133){$p_1$}
	\end{overpic}
	\begin{overpic}[width=6truecm,clip]{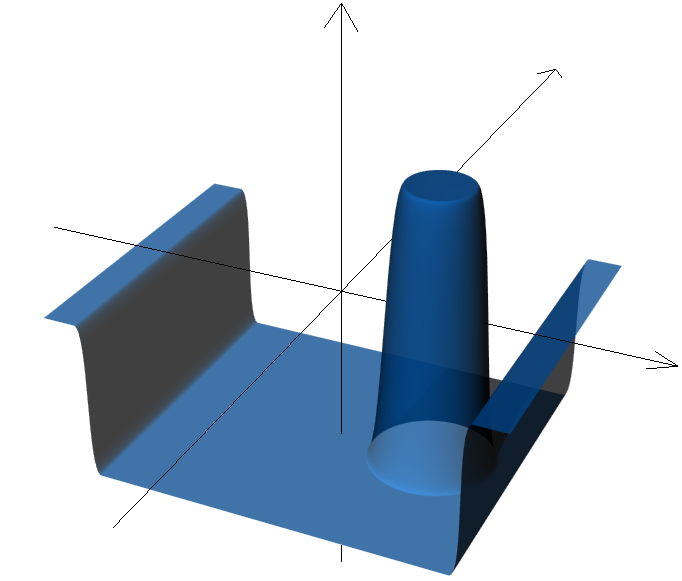}
	\put(78,76){$0$}
	\put(164,45){$p_0$}
	\put(140,133){$p_1$}
	\end{overpic}
\captionof{figure}{Outline of the graphs of $H_s$ (for $s\geq 1$ and $s\leq -1$) in the case $n=1$.}\label{figure:2}
\end{center}
Then for $s\geq 1$ the corresponding Hamiltonian vector field $X_{H_s}$ is of the form
\[
	X_{H_s}(p,q)=%
	\displaystyle\frac{1}{R}f'_s\left(\frac{p_0}{R}\right)\frac{\partial}{\partial q_0},
\]
and for $s\leq -1$
\[
	X_{H_s}(p,q) =%
	\begin{cases}
		\sum\limits_{i=0}^n\frac{1}{m_u}f'_s\left(\frac{\left\| p-\overline{u}\right\|}{m_u}\right)\frac{p_i-\overline{u}_i}{\left\| p-\overline{u}\right\|}\frac{\partial}{\partial q_i} & \text{if\: $\left\| p-\overline{u}\right\|\leq \frac{m_u}{2}$}, \\
		\frac{1}{R-\lvert u\rvert}f'_s\left(\frac{p_0-\lvert u\rvert}{R-\lvert u\rvert}\right)\frac{\partial}{\partial q_0} & \text{if\: $p_0\geq \frac{R+\lvert u\rvert}{2}$}, \\
		\frac{1}{R-\lvert u\rvert}f'_s\left(\frac{p_0+\lvert u\rvert}{R-\lvert u\rvert}\right)\frac{\partial}{\partial q_0} & \text{if\: $p_0\leq -\frac{R+\lvert u\rvert}{2}$}, \\
		0 & \text{otherwise}.
	\end{cases}
\]
For $\ell\in\mathbb{Z}_{>0}$ let us now consider the nontrivial free homotopy class
\[
	\hat{\alpha}_{\ell}=(\alpha_{\ell},0_{\mathbb{T}^{2n}})\in [S^1,A_R\times \mathbb{T}^{2n}].
\]
Then for $s\geq 1$ the set $\mathcal{P}(H_s;\hat{\alpha}_{\ell})$ of periodic trajectories in $\hat{\alpha}_{\ell}$ is denoted by
\[
	\mathcal{P}(H_s;\hat{\alpha}_{\ell})=\left\{\,t\mapsto \bigl(p(t),q(t)\bigr)\relmiddle|p(t), q(t)\ \text{satisfy}\ \eqref{eq:hamsys+}\,\right\},
\]
where
\begin{equation}\label{eq:hamsys+}
	\begin{cases}
		q_0(t)=c_0+\ell t\text{, where}\ c_0\in\mathbb{R}/\mathbb{Z}, \\
		q_1\equiv c_1,\ q_2\equiv c_2,\ \cdots,q_n\equiv c_n\in \mathbb{R}/\mathbb{Z}, \\
		p_0(t)=Rr\text{, where}\ r<0\ \text{is such that}\ f'_s(r)=R\ell, \\
		p_1\equiv d_1,\ p_2\equiv d_2,\ \cdots,p_n\equiv d_n\in \mathbb{R}/2\mathbb{Z}
	\end{cases}
\end{equation}
and for $s\leq -1$
\[
	\mathcal{P}(H_s;\hat{\alpha}_{\ell})=\left\{\,t\mapsto \bigl(p(t),q(t)\bigr)\relmiddle|%
	\begin{aligned}%
		p(t), q(t)\ \text{satisfy}\
		&\eqref{eq:hamsys-a}\ \text{if\: $\left\| p(t)-\overline{u}\right\|\leq \frac{m_u}{2}$}\ (\forall t),\\%
		\text{and}\ &\eqref{eq:hamsys-b}\ \text{if\: $p_0(t)\geq \frac{R+\lvert u\rvert}{2}$}\ (\forall t)%
	\end{aligned}\,\right\},
\]
where
\begin{equation}\label{eq:hamsys-a}
	\begin{cases}
		q_0(t)=c_0+\ell t\text{, where}\ c_0\in\mathbb{R}/\mathbb{Z}, \\
		q_1\equiv c_1,\ q_2\equiv c_2,\ \cdots,q_n\equiv c_n\in \mathbb{R}/\mathbb{Z}, \\
		p_0(t)=u+m_ur\text{, where}\ r<0\ \text{is such that}\ f'_s(r)=m_u\ell, \\
		p_1=p_2=\cdots =p_n=0,
	\end{cases}
\end{equation}
and
\begin{equation}\label{eq:hamsys-b}
	\begin{cases}
		q_0(t)=c_0+\ell t\text{, where}\ c_0\in\mathbb{R}/\mathbb{Z}, \\
		q_1\equiv c_1,\ q_2\equiv c_2,\ \cdots,q_n\equiv c_n\in \mathbb{R}/\mathbb{Z}, \\
		p_0(t)=\lvert u\rvert+(R-\lvert u\rvert)r\text{, where}\ r>0\ \text{is such that}\ f'_s(r)=(R-\lvert u\rvert)\ell, \\
		p_1\equiv d_1,\ p_2\equiv d_2,\ \cdots,p_n\equiv d_n\in \mathbb{R}/2\mathbb{Z}.
	\end{cases}
\end{equation}
Note that there are no periodic trajectories representing $\hat{\alpha}_{\ell}$ if $p_0\leq -(R+\lvert u\rvert)/2$ and $s\leq -1$.
Given $r<0$ with $f_s'(r)=R\ell$, we denote
\[
	\mathcal{P}(r,\hat{\alpha}_{\ell})=\left\{\,x\colon S^1\to A_R\times \mathbb{T}^{2n} \relmiddle| p(t), q(t)\ \text{satisfy}\ \eqref{eq:hamsys+}\,\right\}.
\]
Given $r<0$ with $f'_s(r)=m_u\ell$, we denote
\[
	\mathcal{Q}(r,\hat{\alpha}_{\ell})=\left\{\,x\colon S^1\to A_R\times \mathbb{T}^{2n}\relmiddle| p(t), q(t)\ \text{satisfy}\ \eqref{eq:hamsys-a}\,\right\}.
\]
Given $r>0$ with $f'_s(r)=(R-\lvert u\rvert)\ell$, we denote
\[
	\mathcal{R}(r,\hat{\alpha}_{\ell})=\left\{\,x\colon S^1\to A_R\times \mathbb{T}^{2n}\relmiddle| p(t), q(t)\ \text{satisfy}\ \eqref{eq:hamsys-b}\,\right\}.
\]
Then $\mathcal{P}(r,\hat{\alpha}_{\ell})$ and $\mathcal{R}(r,\hat{\alpha}_{\ell})$ are diffeomorphic to $\mathbb{T}^{2n+1}$,
and $\mathcal{Q}(r,\hat{\alpha}_{\ell})$ is diffeomorphic to $\mathbb{T}^{n+1}$.
In summary, we have
\[
	\mathcal{P}(H_s;\hat{\alpha}_{\ell})=%
	\begin{cases}
		\mathcal{P}(r_s,\hat{\alpha}_{\ell})\sqcup\mathcal{P}(r'_s,\hat{\alpha}_{\ell}) & \text{if\: $s\geq 1$}, \\
		\mathcal{Q}(r_s,\hat{\alpha}_{\ell})\sqcup\mathcal{Q}(r'_s,\hat{\alpha}_{\ell})\sqcup\bigsqcup_r\mathcal{R}(r,\hat{\alpha}_{\ell}) & \text{if\: $s\leq -1$},
	\end{cases}
\]
where the union $\bigsqcup_r\mathcal{R}(r,\hat{\alpha}_{\ell})$ runs over $r>0$ such that $f_s'(r)=(R-\lvert u\rvert)\ell$.

\begin{lemma}[{\cite[Lemma 5.3.2]{BPS}}]\label{lemma:MB}
The sets $\mathcal{P}(r,\hat{\alpha}_{\ell})$, $\mathcal{Q}(r,\hat{\alpha}_{\ell})$ and $\mathcal{R}(r,\hat{\alpha}_{\ell})$ are Morse--Bott manifolds of periodic trajectories for $H_s$ if and only if $f''_s(r)\neq 0$.
\end{lemma}

We claim that the value of the action functional on $\mathcal{R}(r,\hat{\alpha}_{\ell})$ is negative.
In fact, since we have $f'_s(r)=(R-\lvert u\rvert)\ell$ for any periodic trajectory in $\mathcal{R}(r,\hat{\alpha}_{\ell})$,
the action of such a periodic trajectory is $f_s(r)-\lvert u\rvert\ell-(R-\lvert u\rvert)r\ell$,
and this is negative by (viii).
On the other hand, by (vi), (vii) and Lemma \ref{lemma:MB}, $\mathcal{P}(r,\hat{\alpha}_{\ell})$ and $\mathcal{Q}(r,\hat{\alpha}_{\ell})$ are Morse--Bott manifolds of periodic trajectories for $H_s$
and the values of the action functional on these critical manifolds are
\[
	\mathcal{A}_{H_s}\bigl(\mathcal{P}(r,\hat{\alpha}_{\ell})\bigr)=f_s(r)-Rr\ell,
\]
and
\[
	\mathcal{A}_{H_s}\bigl(\mathcal{Q}(r,\hat{\alpha}_{\ell})\bigr)=f_s(r)-(u+m_ur)\ell,
\]
respectively.
Fix a real number $a$.
We prove Theorem \ref{theorem:main} in four steps.


\begin{step}\label{step:1}
If $a<R\ell$, then ${\SH}^{[a,\infty)}(A_R\times \mathbb{T}^{2n};\hat{\alpha}_{\ell})=0$.
\end{step}

\begin{center}
	\begin{overpic}[width=8truecm,clip]{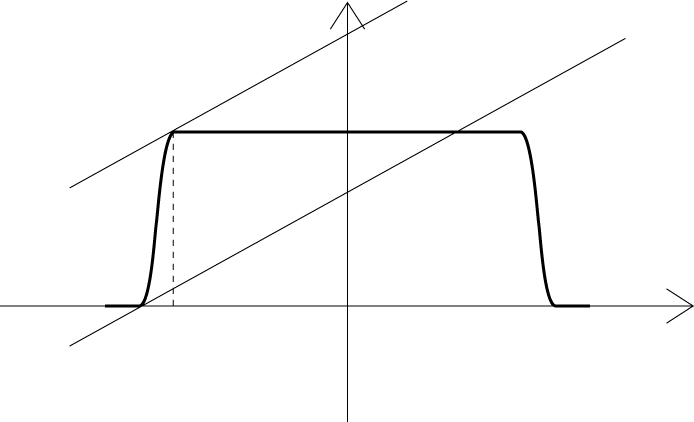}
	\put(116,29){$0$}
	\put(191,29){$R$}
	\put(24,29){$-R$}
	\put(230,36){$p_0$}
	\put(50,29){$Rr_s$}
	\put(43,29){$\uparrow$}
	\put(38,19){$Rr_s'$}
	\put(8,120){tangent at $p_0=Rr_s$}
	\put(190,110){tangent at $p_0=Rr_s'$}
	\end{overpic}
\captionof{figure}{Outline of the graph of $H_s$ (for $s\geq 1$) in the direction of $p_0$.}\label{figure:3}
\end{center}

For $s\geq 1$ the only families of periodic trajectories are $\mathcal{P}(r_s,\hat{\alpha}_{\ell})$ and $\mathcal{P}(r'_s,\hat{\alpha}_{\ell})$.
Since both $r_s$ and $r'_s$ converge to $-1$ as $s\to\infty$,
the values of the action functional on $\mathcal{P}(r_s,\hat{\alpha}_{\ell})$ and $\mathcal{P}(r'_s,\hat{\alpha}_{\ell})$ are both bigger than $a$ for $s$ sufficiently large.
Hence
\[
	\mathrm{HF}^{[a,\infty)}(H_s;\hat{\alpha}_{\ell})\cong \mathrm{HF}^{[-\infty,\infty)}(H_s;\hat{\alpha}_{\ell})=0
\]
for $s$ sufficiently large.
Here since $\mathrm{HF}^{[-\infty,\infty)}(H_s;\hat{\alpha}_{\ell})$ is independent of $H_s$ (see \cite[Proposition 4.5.1]{BPS})
and every $C^2$-small Hamiltonian has only contractible periodic trajectories,
the last equation holds.
Now Step 1 follows from \cite[Lemma 4.7.1 (ii)]{BPS}.
 

\begin{step}\label{step:2}
If $a\geq R\ell$, then ${\SH}^{[a,\infty)}(A_R\times \mathbb{T}^{2n};\hat{\alpha}_{\ell})\cong H_{\ast}(\mathbb{T}^{2n+1};\mathbb{Z}/2\mathbb{Z})$.
Moreover, the homomorphism
\[
	\pi_{H_s}\colon {\SH}^{[a,\infty)}(A_R\times \mathbb{T}^{2n};\hat{\alpha}_{\ell})\to \mathrm{HF}^{[a,\infty)}(H_s;\hat{\alpha}_{\ell})
\]
is an isomorphism whenever $f_s(0)>a$.
\end{step}

By (vi), for any $s\geq 1$
the number $r_s$ (resp.\ $r'_s$) is the maximum point (resp.\ the minimum point)
of the function $f_{s,\ell}\colon [-1,0]\to\mathbb{R}$ defined by
\[
	f_{s,\ell}(r)=f_s(r)-Rr\ell.
\]
If $s$ is sufficiently large so that $f_s(0)>a$, then $f_{s,\ell}(0)=f_s(0)>a$ and hence
\[
	\mathcal{A}_{H_s}\bigl(\mathcal{P}(r_s,\hat{\alpha}_{\ell})\bigr)=f_{s,\ell}(r_s)>f_{s,\ell}(0)>a.
\]
Since $f_{s,\ell}(-1)=R\ell\leq a$, we have
\[
	\mathcal{A}_{H_s}\bigl(\mathcal{P}(r'_s,\hat{\alpha}_{\ell})\bigr)=f_{s,\ell}(r'_s)<f_{s,\ell}(-1)\leq a.
\]
Hence, by Theorem \ref{theorem:Pozniak},
$\mathrm{HF}^{[a,\infty)}(H_s;\hat{\alpha}_{\ell})\cong H_{\ast}(\mathbb{T}^{n+1};\mathbb{Z}/2\mathbb{Z})$
and, by \cite[Proposition 4.5.1]{BPS}, the monotone homomorphism
\[
	\sigma_{H_{s_1}H_{s_0}}\colon \mathrm{HF}^{[a,\infty)}(H_{s_0};\hat{\alpha}_{\ell})\to\mathrm{HF}^{[a,\infty)}(H_{s_1};\hat{\alpha}_{\ell})
\]
is an isomorphism whenever $f_{s_i}(0)>a$ for $i=0,1$ and $s_1\leq s_0$ and Step 2 follows from \cite[Lemma 4.7.1 (ii)]{BPS}.


\begin{step}\label{step:3}
If $a>c-u\ell>0$, then ${\RSH}^{[a,\infty);c}(A_R\times \mathbb{T}^{2n},L_u\times \mathbb{T}^n;\hat{\alpha}_{\ell})=0$.
\end{step}

\begin{center}
	\begin{overpic}[width=8truecm,clip]{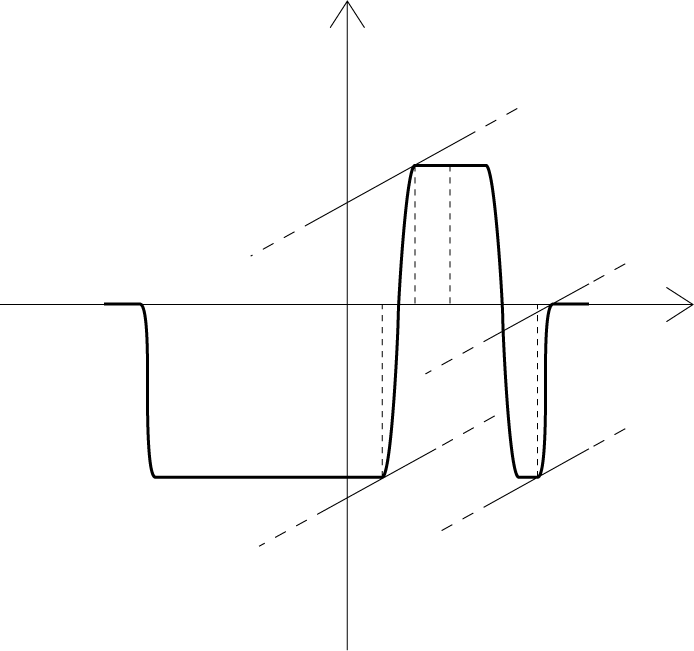}
	\put(113,155){\line(1,0){2}}
	\put(116,104){$0$}
	\put(116,150){$c$}
	\put(191,104){$R$}
	\put(25,104){$-R$}
	\put(230,111){$p_0$}
	\put(145,104){$u$}
	\put(140,180){tangent at $p_0=u+m_ur_s$}
	\put(-5,26){tangent at $p_0=u+m_ur_s'$}
	\put(208,125){$L_1$}
	\put(208,70){$L_2$}
	\end{overpic}
\captionof{figure}{Outline of the graph of $H_s$ (for $s\leq -1$) in the direction of $p_0$.}\label{figure:4}
\end{center}

In Figure \ref{figure:4}, the lines $L_1,L_2$ are tangent to $H_s$ at $p_0=\lvert u\rvert+(R-\lvert u\rvert)r$,
where $r>0$ is such that $f_s'(r)=(R-\lvert u\rvert)\ell$ (if such $r>0$ exists).
Since $a>0$, we can ignore all trajectories which belong to $\mathcal{R}(r,\hat{\alpha}_{\ell})$.
Since both $r_s$ and $r'_s$ converge to 0 as $s\to -\infty$,
the values of the action functional on $\mathcal{Q}(r_s,\hat{\alpha}_{\ell})$ and $\mathcal{Q}(r'_s,\hat{\alpha}_{\ell})$ are both less than $a$ for $-s$ sufficiently large.
Hence $\mathrm{HF}^{[a,\infty)}(H_s;\hat{\alpha}_{\ell})=0$ for $-s$ sufficiently large.
Now Step 3 follows from \cite[Lemma 4.7.1 (i)]{BPS}.


\begin{step}\label{step:4}
If $0<a\leq c-u\ell$, then ${\RSH}^{[a,\infty);c}(A_R\times \mathbb{T}^{2n},L_u\times \mathbb{T}^n;\hat{\alpha}_{\ell})\cong H_{\ast}(\mathbb{T}^{n+1};\mathbb{Z}/2\mathbb{Z})$.
Moreover, the homomorphism
\[
	\iota_{H_s}\colon\mathrm{HF}^{[a,\infty)}(H_s;\hat{\alpha}_{\ell})\to{\RSH}^{[a,\infty);c}(A_R\times \mathbb{T}^{2n},L_u\times \mathbb{T}^n;\hat{\alpha}_{\ell})
\]
is an isomorphism for $s\ll -1$.
\end{step}

As $a>0$ we may ignore trajectories which belong to $\mathcal{R}(r,\hat{\alpha}_{\ell})$.
By (vii), for any $s\leq -1$
the number $r_s$ (resp.\ $r'_s$) is the maximum point (resp.\ the minimum point)
of the function $g_{s,\ell}\colon [-1,0]\to\mathbb{R}$ defined by
\[
	g_{s,\ell}(r)=f_s(r)-(u+m_ur)\ell.
\]
Then, by (ii), $g_{s,\ell}(0)=f_s(0)-u\ell>c-u\ell\geq a$ and hence
\[
	\mathcal{A}_{H_s}\bigl(\mathcal{Q}(r_s,\hat{\alpha}_{\ell})\bigr)=g_{s,\ell}(r_s)>g_{s,\ell}(0)>a.
\]
If $s<\min\{-1,a+(u-m_u/2)\ell\}$, then $g_{s,\ell}(-1/2)=s-(u-m_u/2)\ell<a$ and hence
\[
	\mathcal{A}_{H_s}\bigl(\mathcal{Q}(r'_s,\hat{\alpha}_{\ell})\bigr)=g_{s,\ell}(r'_s)<g_{s,\ell}(-1/2)<a.
\]
Applying Theorem \ref{theorem:Pozniak}, $\mathrm{HF}^{[a,\infty)}(H_s;\hat{\alpha}_{\ell})\cong H_{\ast}(\mathbb{T}^{n+1};\mathbb{Z}/2\mathbb{Z})$
for $s<\min\{-1,a+(u-m_u/2)\ell\}$.
By \cite[Proposition 4.5.1]{BPS}, the monotone homomorphism
\[
	\sigma_{H_{s_1}H_{s_0}}\colon \mathrm{HF}^{[a,\infty)}(H_{s_0};\hat{\alpha}_{\ell})\to\mathrm{HF}^{[a,\infty)}(H_{s_1};\hat{\alpha}_{\ell})
\]
is an isomorphism for $s_1<s_0<\min\{-1,a+(u-m_u/2)\ell\}$.
Thus Step 4 follows from \cite[Lemma 4.7.1 (i)]{BPS}.
The proof of Theorem \ref{theorem:main} is complete.
\end{proof}


\subsection{Proof of Theorem \ref{theorem:main2}}

\begin{proof}
If the homomorphism
\[
	\left(T_{\hat{\alpha}_{\ell}}^{[a,\infty);c}\right)_k\colon {\SH}{}_k^{[a,\infty)}(A_R\times \mathbb{T}^{2n};\hat{\alpha}_{\ell})%
	\to {\RSH}{}_k^{[a,\infty);c}(A_R\times \mathbb{T}^{2n},L_u\times \mathbb{T}^n;\hat{\alpha}_{\ell})
\]
is non-zero,
then Theorem \ref{theorem:main0}, Step \ref{step:1} and Step \ref{step:3} in the proof of Theorem \ref{theorem:main} imply that $R\lvert\ell\rvert<a\leq c-u\ell$.

Fix $R\lvert\ell\rvert<a\leq c-u\ell$.
For simplicity, we assume that $\ell\neq0$.
The proof for the case $\ell=0$ follows the same path as in the case $\ell\neq0$.
We choose a sufficiently large $T\gg 1$ so that for all $k$ the homomorphisms
\[
	\left(\pi_{H_T}\right)_k\colon {\SH}{}_k^{[a,\infty)}(A_R\times \mathbb{T}^{2n};\hat{\alpha}_{\ell})\xrightarrow{\cong} \mathrm{HF}_k^{[a,\infty)}(H_T;\hat{\alpha}_{\ell})
\]
and
\[
	\left(\iota_{H_{-T}}\right)_k\colon\mathrm{HF}_k^{[a,\infty)}(H_{-T};\hat{\alpha}_{\ell})\xrightarrow{\cong} {\RSH}{}_k^{[a,\infty);c}(A_R\times \mathbb{T}^{2n},L_u\times \mathbb{T}^n;\hat{\alpha}_{\ell})
\]
are isomorphisms.
By Proposition \ref{proposition:diagram}, we have the following commutative diagram.
\[
	\xymatrix{
	{\SH}{}_k^{[a,\infty)}(A_R\times \mathbb{T}^{2n};\hat{\alpha}_{\ell}) \ar[d]_{\left(\pi_{H_T}\right)_k}^{\cong} \ar[rr]^(0.44){\left(T_{\hat{\alpha}_{\ell}}^{[a,\infty);c}\right)_k} & & {\RSH}{}_k^{[a,\infty);c}(A_R\times \mathbb{T}^{2n},L_u\times \mathbb{T}^n;\hat{\alpha}_{\ell}) \\
	\mathrm{HF}_k^{[a,\infty)}(H_T;\hat{\alpha}_{\ell}) \ar[rr]^{\left(\sigma_{H_{-T}H_T}\right)_k} & & \mathrm{HF}_k^{[a,\infty)}(H_{-T};\hat{\alpha}_{\ell}) \ar[u]_{\left(\iota_{H_{-T}}\right)_k}^{\cong} \\
	}
\]
Hence it is enough to show that the homomorphism $\left(\sigma_{H_{-T}H_T}\right)_k$ is non-zero and surjective.

We put $C_T=\mathrm{ev}(\mathcal{P}(r_T,\hat{\alpha}_{\ell}))\cong\mathbb{T}^{2n+1}$
and $C_{-T}=\mathrm{ev}(\mathcal{Q}(r_T,\hat{\alpha}_{\ell}))\cong \mathbb{T}^{n+1}$,
where $\mathrm{ev}$ is the evaluation map given by $\mathrm{ev}(x)=x(0)$.
We define a Morse function $F_{-T}\colon C_{-T}\to\mathbb{R}$ by
\[
	F_{-T}(q_0,q_1,\ldots,q_n)=-n(n+2)+\frac{1}{2}\left(n+1+\sum_{i=0}^n\cos(2\pi q_i)\right)
\]
for $(q_0,q_1,\ldots,q_n)\in(\mathbb{R}/\mathbb{Z})^{n+1}$ and a Morse function $F_T\colon C_T\to\mathbb{R}$ by
\begin{align*}
	F_T(p_1,\ldots,p_n,q_0,q_1,\ldots,q_n)&=-n(n+2)+\frac{n+2}{2}\left(n+\sum_{i=1}^n\cos(\pi p_i)\right)\\
	&+\frac{1}{2}\left(n+1+\sum_{i=0}^n\cos(2\pi q_i)\right)
\end{align*}
for $(p_1,\ldots,p_n,q_0,q_1,\ldots,q_n)\in(\mathbb{R}/2\mathbb{Z})^n\times(\mathbb{R}/\mathbb{Z})^{n+1}$.
Let $\gamma\in C_T$ be a maximal point of $F_T$ with respect to the coordinate $(p_1,\ldots,p_n)$ and
a minimal point with respect to $(q_0,q_1,\ldots,q_n)$.
Then we have $F_T(\gamma)=0$.

Let $N_{\pm T}\subset A_R\times\mathbb{T}^{2n}$ be tubular neighborhoods of $C_{\pm T}$ of radii $\delta_{\pm T}>0$, respectively.
We extend the functions $F_{\pm T}$ to functions on $N_{\pm T}$
by making constant in the direction normal to $C_{\pm T}$.
Let $\tilde{N}_{\pm T}$ be smaller tubular neighborhoods of $C_{\pm T}$
of radii $\tilde{\delta}_{\pm T}(<\delta_{\pm T})$.
Moreover, let $\hat{\rho}_{\pm T}\colon [0,\infty )\to [0,1]$ be $C^{\infty}$-functions satisfying
\begin{align*}
\hat{\rho}_{\pm T} &=%
\begin{cases}
	1 & \text{on\: $[0,\tilde{\delta}_{\pm T})$},\\
	0 & \text{outside\: $[0,\delta_{\pm T})$},
\end{cases}\\
\hat{\rho}_{\pm T}&>0\quad \text{on\: $[0,\delta_{\pm T})$}
\end{align*}
and we define bump functions $\rho_{\pm T}\colon A_R\times\mathbb{T}^{2n}\to\mathbb{R}$ by $\rho_{\pm T}(x)=\hat{\rho}_{\pm T}\bigl(d(x,C_{\pm T})\bigr)$
where $d$ is the distance induced from a metric on $A_R\times\mathbb{T}^{2n}$.

Let $(M,\omega)$ be a symplectic manifold.
For two Hamiltonians $H,K\colon S^1\times M\to\mathbb{R}$ with compact support, we define a juxtaposition $H\natural K\colon S^1\times M\to\mathbb{R}$ of $H$ and $K$ by
\[
	(H\natural K)(t,x)=%
	\begin{cases}%
		\chi'(t)H(\chi(t),x)\ & \text{if\: $0\leq t\leq 1/2$},\\
		\chi'(t-1/2)K(\chi(t-1/2),x)\ & \text{if\: $1/2\leq t\leq 1$}
	\end{cases}
\]
where $\chi\colon [0,1/2]\to [0,1]$ is a smooth function with the properties that
$\chi'(t)\geq 0$ for all $t$, $\chi(0)=0$, $\chi(1/2)=1$,
and $\chi'$ vanishes to infinite order at $0$ and at $1/2$ (see \cite[Proof of Proposition 3.1]{Us} for details).
Since $\chi'$ vanishes to infinite order at $0$ and $1/2$, $H\natural K$ is one-periodic in time, that is,
that $H\natural K\in C_0^{\infty}(S^1\times M)$.

For a real number $\varepsilon>0$ we consider the two juxtapositions $H_T\natural(\varepsilon \rho_TF_T)$ and $H_{-T}\natural(\varepsilon \rho_{-T}F_{-T})$.
By \cite[Proof of Theorem B]{BPS}, we may assume, without loss of generality,
that $H_T\natural(\varepsilon \rho_TF_T)$ and $H_{-T}\natural(\varepsilon \rho_{-T}F_{-T})$ are one-periodic in time.
Now we choose neighborhoods $\mathcal{U}_{H_T}$ and $\mathcal{U}_{H_{-T}}$ as in Proposition \ref{proposition:nbd}.
We choose a sufficiently small $\varepsilon>0$ so that $\varepsilon<1$, $H_T\natural(\varepsilon \rho_TF_T)\in\mathcal{U}_{H_T}$ and $H_{-T}\natural(\varepsilon \rho_{-T}F_{-T})\in\mathcal{U}_{H_{-T}}$.
We then have
\[
	\mathrm{HF}_k^{[a,\infty)}(H_{\pm T};\hat{\alpha}_{\ell})=\mathrm{HF}_k^{[a,\infty)}(H_{\pm T}\natural(\varepsilon\rho_{\pm T} F_{\pm T});\hat{\alpha}_{\ell}).
\]
Moreover, we note that $\mathcal{A}_{H_T}(\gamma)=\mathcal{A}_{H_T\natural(\varepsilon \rho_TF_T)}(\gamma)$.
These perturbations enable us to compute the Floer homology groups via Po\'{z}niak's theorem (Theorem \ref{theorem:Pozniak})
and reveal their behavior when the action interval $[a,\infty)$ varies.


\begin{claim}\label{claim:1}
For all $k=0,1,\ldots,2n+1$ we have
\[
	\mathrm{HF}_k^{[a,\infty)}(H_T\natural(\varepsilon \rho_TF_T);\hat{\alpha}_{\ell})\cong%
	H_k(\mathbb{T}^{2n+1};\mathbb{Z}/2\mathbb{Z}).
\]
\end{claim}

\begin{proof}
By Step \ref{step:2}, we have
\[
	\mathcal{P}^{[a,\infty)}(H_T;\hat{\alpha}_{\ell})=\mathcal{P}(r_T,\hat{\alpha}_{\ell}).
\]
By Lemma \ref{lemma:MB}, $\mathcal{P}^{[a,\infty)}(H_T;\hat{\alpha}_{\ell})$ is a Morse--Bott manifold of periodic trajectories for $H_T$.
By Theorem \ref{theorem:Pozniak}, the Floer homology $\mathrm{HF}_k^{[a,\infty)}(H_T;\hat{\alpha}_{\ell})$ coincides with
the Morse homology $\mathrm{HM}_{2n+1-k}\bigl(\mathcal{P}(r_T;\hat{\alpha}_{\ell}),F_T\bigr)$.
Therefore we have
\[
	\mathrm{HF}_k^{[a,\infty)}(H_T\natural(\varepsilon \rho_TF_T);\hat{\alpha}_{\ell})%
	\cong \mathrm{HM}_{2n+1-k}\bigl(\mathcal{P}(r_T,\hat{\alpha}_{\ell}),F_T\bigr)%
	\cong H_k(\mathbb{T}^{2n+1};\mathbb{Z}/2\mathbb{Z}).\qedhere
\]
\end{proof}

Thus for any $x\in\mathcal{P}(r_T,\hat{\alpha}_{\ell})$,
\[
	a<\mathcal{A}_{H_T\natural(\varepsilon \rho_TF_T)}(x)%
	=\mathcal{A}_{H_T}\bigl(\mathcal{P}(r_T,\hat{\alpha}_{\ell})\bigr)+\varepsilon F_T(x).
\]
We choose a positive real number $b$ such that
\[
	\mathcal{A}_{H_T}\bigl(\mathcal{P}(r_T,\hat{\alpha}_{\ell})\bigr)+\varepsilon\cdot(-1)<b<\mathcal{A}_{H_T}\bigl(\mathcal{P}(r_T,\hat{\alpha}_{\ell})\bigr)+\varepsilon F_T(\gamma)%
	=\mathcal{A}_{H_T}\bigl(\mathcal{P}(r_T,\hat{\alpha}_{\ell})\bigr).
\]
Since the minimum value of $F_T$ is $-n(n+2)$,
we obtain
\[
	a<\mathcal{A}_{H_T}\bigl(\mathcal{P}(r_T,\hat{\alpha}_{\ell})\bigr)-\varepsilon n(n+2)<b.
\]
\begin{center}
	\begin{overpic}[width=8truecm,clip]{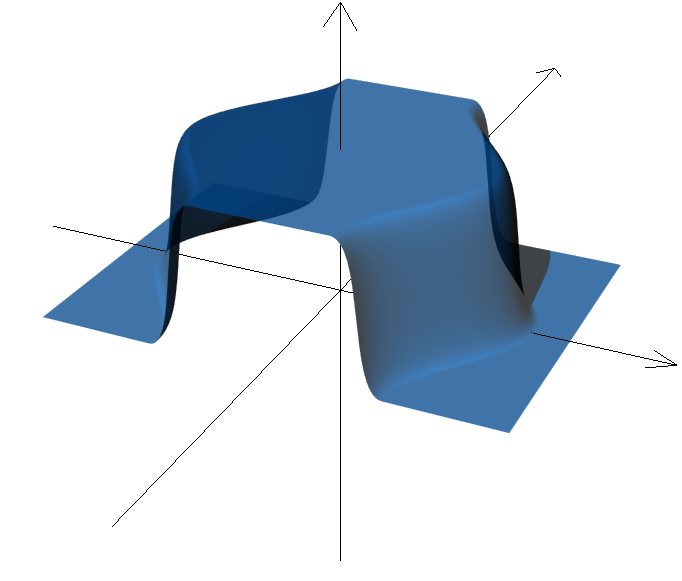}
	\put(107,84){$0$}
	\put(229,69){$p_0$}
	\put(185,176){$p_1$}
	\end{overpic}
\captionof{figure}{Outline of the graph of $H_T\natural(\varepsilon \rho_TF_T)$ in the case $n=1$.}\label{figure:5}
\end{center}

We deform $H_T\natural(\varepsilon \rho_TF_T)|_{A_R\times \mathbb{T}^{2n}\setminus U}$
small enough so that every periodic trajectory $x$ lying in $A_R\times \mathbb{T}^{2n}\setminus U$ has action less than $a$
(i.e., the tangent line at $x$ of slope $\ell$ does not take values greater than $a$ at $p_0=0$),
where $U$ is a small open neighborhood of $\{\,p_1=\cdots=p_n=0\,\}$.
We then obtain a Hamiltonian $\widetilde{H}\colon A_R\times \mathbb{T}^{2n}\to\mathbb{R}$ with compact support so that
\[
	\widetilde{H}=H_T\natural(\varepsilon \rho_TF_T)\quad \text{near}\ p_1=\cdots=p_n=0,
\]
\[
	H_T\natural(\varepsilon \rho_TF_T)\preceq\widetilde{H}\preceq H_{-T}\natural(\varepsilon \rho_{-T}F_{-T})
\]
and the Hamiltonian isotopy associated to $\widetilde{H}$ does not admit periodic trajectories of action greater than $a$ and less than $b$.
\begin{center}
	\begin{overpic}[width=8truecm,clip]{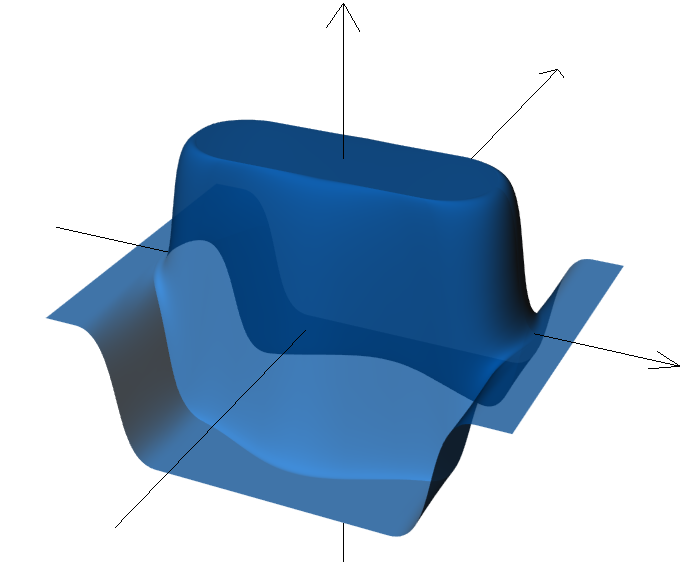}
	\put(229,69){$p_0$}
	\put(185,176){$p_1$}
	\end{overpic}
\captionof{figure}{Outline of the graph of $\widetilde{H}$ in the case $n=1$.}\label{figure:6}
\end{center}
We may assume, without loss of generality, that our monotone homotopy $\{H_s\}_s$ contains $H_T\natural(\varepsilon \rho_TF_T)$, $\widetilde{H}$ and $H_{-T}\natural(\varepsilon \rho_{-T}F_{-T})$.
By \eqref{eq:longexactseq} in Subsection \ref{subsection:monotone},
we have the following commutative diagram,
whose rows are the short exact sequences for $\widetilde{H}_T=H_T\natural(\varepsilon \rho_TF_T)$, for $\widetilde{H}$ and for $\widetilde{H}_{-T}=H_{-T}\natural(\varepsilon \rho_{-T}F_{-T})$.
\begin{equation}\label{eq:diagram}
\xymatrix{
\cdots\ar[r] & \mathrm{HF}_{\ast}^{[a,b)}(\widetilde{H}_T) \ar[d]^{\sigma_{\widetilde{H}\widetilde{H}_T}} \ar[r]^{[\iota^F]}  & \mathrm{HF}_{\ast}^{[a,\infty)}(\widetilde{H}_T) \ar[d]^{\sigma_{\widetilde{H}\widetilde{H}_T}} \ar[r]^{[\pi^F]} & \mathrm{HF}_{\ast}^{[b,\infty)}(\widetilde{H}_T) \ar[r] \ar[d]^{\sigma_{\widetilde{H}\widetilde{H}_T}}_{\cong} & \cdots \\
\cdots\ar[r] & \mathrm{HF}_{\ast}^{[a,b)}(\widetilde{H}) \ar[d]^{\sigma_{\widetilde{H}_{-T}\widetilde{H}}} \ar[r]^{[\iota^F]}  & \mathrm{HF}_{\ast}^{[a,\infty)}(\widetilde{H}) \ar[d]^{\sigma_{\widetilde{H}_{-T}\widetilde{H}}}_{\cong} \ar[r]^{[\pi^F]}_{\cong} & \mathrm{HF}_{\ast}^{[b,\infty)}(\widetilde{H}) \ar[r] \ar[d]^{\sigma_{\widetilde{H}_{-T}\widetilde{H}}} & \cdots \\
\cdots\ar[r] & \mathrm{HF}_{\ast}^{[a,b)}(\widetilde{H}_{-T}) \ar[r]^{[\iota^F]} & \mathrm{HF}_{\ast}^{[a,\infty)}(\widetilde{H}_{-T}) \ar[r]^{[\pi^F]} & \mathrm{HF}_{\ast}^{[b,\infty)}(\widetilde{H}_{-T}) \ar[r] & \cdots \\
}
\end{equation}
Here we temporarily omitted $\hat{\alpha}_{\ell}$ in the notation.
By Definition \ref{definition:degenerate} and Lemma \ref{lemma:functoriality}, note that
\begin{equation}\label{eq:composition}
	\sigma_{\widetilde{H}_{-T}\widetilde{H}}\circ\sigma_{\widetilde{H}\widetilde{H}_T}=\sigma_{H_{-T}H_T}.
\end{equation}


\begin{claim}\label{claim:2}
For all $k=0,1,\ldots,n+1$ we have
\[
	\mathrm{HF}_k^{[b,\infty)}(\widetilde{H}_T;\hat{\alpha}_{\ell})\cong%
	H_k(\mathbb{T}^{n+1};\mathbb{Z}/2\mathbb{Z}).
\]
\end{claim}

\begin{proof}
By the choice of $b$, we obtain
\[
	\mathcal{P}^{[b,\infty)}(H_T;\hat{\alpha}_{\ell})\cong S:=\{\,F_T(p_1,\ldots,p_n,q_0,q_1,\ldots,q_n)>-1\,\}\subset\mathbb{T}^{2n}.
\]
Let $c=(p_1,\ldots,p_n,q_0,q_1,\ldots,q_n)\in S$ be a critical point of the restriction $F_T|_S$.
Then each $\cos{(\pi p_i)}$ and $\cos{(2\pi q_i)}$ take values in $\{1,-1\}$.
We claim that for any $i=1,\dots,n$ we have $\cos{(\pi p_i)}=1$.
Actually, if $\cos{(\pi p_{i_0})}=-1$ for some $i_0$ (we may assume, without loss of generality, that $i_0=n$),
then we have
\begin{align*}
	F_T(c)%
	&=-n(n+2)+\frac{n+2}{2}\left(n+\sum_{i=1}^{n-1}\cos(\pi p_i)-1\right)+\frac{1}{2}\left(n+1+\sum_{i=0}^n\cos(2\pi q_i)\right)\\
	&=-\frac{(n+1)^2}{2}+\frac{n+2}{2}\sum_{i=1}^{n-1}\cos(\pi p_i)+\frac{1}{2}\sum_{i=0}^n\cos(2\pi q_i)\\
	&\leq-\frac{(n+1)^2}{2}+\frac{n+2}{2}(n-1)+\frac{1}{2}(n+1)%
	=-1.
\end{align*}
This contradicts the fact that $F_T(c)>-1$.
Hence $S$ is homeomorphic to the product of a small open $n$-cell $e^n$ and the $(n+1)$-torus $\mathbb{T}^{n+1}$.
Theorem \ref{theorem:Pozniak} and Claim \ref{claim:1} show that the Floer homology $\mathrm{HF}_k^{[b,\infty)}(H_T;\hat{\alpha}_{\ell})$
coincides with the Morse homology $\mathrm{HM}_{2n+1-k}\bigl(\mathcal{P}^{[b,\infty)}(H_T;\hat{\alpha}_{\ell}),F_T\bigr)$
even if the action interval varies.
Hence we have
\begin{align*}
	\mathrm{HF}_k^{[b,\infty)}(\widetilde{H}_T;\hat{\alpha}_{\ell})%
	&\cong \mathrm{HM}_{2n+1-k}(e^n\times \mathbb{T}^{n+1},F_T) \\
	&\cong \bigoplus_{i+j=2n+1-k} \left(H_i(e^n,\partial e^n)\otimes_{\mathbb{Z}/2\mathbb{Z}} H_j(\mathbb{T}^{n+1})\right) \\
	&\cong H_n(S^n)\otimes_{\mathbb{Z}/2\mathbb{Z}} H_{n+1-k}(\mathbb{T}^{n+1}) \\
	&\cong H_k(\mathbb{T}^{n+1};\mathbb{Z}/2\mathbb{Z}).\qedhere
\end{align*}
\end{proof}


\begin{claim}\label{claim:3}
The homomorphism
\[
	\sigma_{\widetilde{H}\widetilde{H}_T}\colon \mathrm{HF}^{[b,\infty)}(\widetilde{H}_T;\hat{\alpha}_{\ell})\to%
	\mathrm{HF}^{[b,\infty)}(\widetilde{H};\hat{\alpha}_{\ell})
\]
is an isomorphism.
\end{claim}

\begin{proof}
We can choose a monotone homotopy $\{H_s\}_s$ defining the map $\sigma_{\widetilde{H}\widetilde{H}_T}$ so that
$H_s$ does not allow new periodic trajectories of action greater than $b$ for all $s$,
i.e., $\{H_s\}_s$ is action-regular.
Hence Lemma \ref{lemma:actionregular} shows that $\sigma_{\widetilde{H}\widetilde{H}_T}$ is an isomorphism.
\end{proof}


\begin{claim}\label{claim:4}
The homomorphism
\[
	[\pi^F]\colon \mathrm{HF}^{[a,\infty)}(\widetilde{H};\hat{\alpha}_{\ell})\to%
	\mathrm{HF}^{[b,\infty)}(\widetilde{H};\hat{\alpha}_{\ell})
\]
is an isomorphism.
\end{claim}

\begin{proof}
Since the Hamiltonian isotopy associated to $\widetilde{H}$ does not admit periodic trajectories of action greater than $a$ and less than $b$,
we have $\mathrm{HF}^{[a,b)}(\widetilde{H};\hat{\alpha}_{\ell})=0$.
The exactness of the second row in \eqref{eq:diagram} shows Claim \ref{claim:4}.
\end{proof}


\begin{claim}\label{claim:5}
The homomorphism
\[
	\sigma_{\widetilde{H}_{-T}\widetilde{H}}\colon \mathrm{HF}^{[a,\infty)}(\widetilde{H};\hat{\alpha}_{\ell})\to%
	\mathrm{HF}^{[a,\infty)}(\widetilde{H}_{-T};\hat{\alpha}_{\ell})
\]
is an isomorphism.
\end{claim}

\begin{proof}
A similar observation in the proof of Claim \ref{claim:3} shows Claim \ref{claim:5}.
\end{proof}

Actually, Step \ref{step:4} directly shows that
$\mathrm{HF}^{[a,\infty)}(\widetilde{H}_{-T};\hat{\alpha}_{\ell})\cong H_{\ast}(\mathbb{T}^{n+1};\mathbb{Z}/2\mathbb{Z})$.
By Claim \ref{claim:3}--\ref{claim:5} and \eqref{eq:composition}, we deduce that the homomorphism
\[
	\left(\sigma_{H_{-T}H_T}\right)_k\colon \mathrm{HF}_k^{[a,\infty)}(H_T;\hat{\alpha}_{\ell})\to \mathrm{HF}_k^{[a,\infty)}(H_{-T};\hat{\alpha}_{\ell})
\]
is non-zero and surjective if and only if so is the homomorphism
\[
	[\pi^F]_k\colon \mathrm{HF}_k^{[a,\infty)}(\widetilde{H}_T;\hat{\alpha}_{\ell})\to%
	\mathrm{HF}_k^{[b,\infty)}(\widetilde{H}_T;\hat{\alpha}_{\ell}).
\]
Therefore,
the exactness of the first row in \eqref{eq:diagram} implies that
it is suffices, for proving the non-zeroness of $[\pi^F]_k$, to show that the homomorphism
\[
	[\iota^F]_k\colon \mathrm{HF}_k^{[a,b)}(\widetilde{H}_T;\hat{\alpha}_{\ell})\to \mathrm{HF}_k^{[a,\infty)}(\widetilde{H}_T;\hat{\alpha}_{\ell})
\]
is not surjective if $k=0,1,\ldots,n+1$.


\begin{claim}\label{claim:6}
For all $k=0,1,\ldots,2n+1$ we have
\[
	\mathrm{HF}_k^{[a,b)}(\widetilde{H}_T;\hat{\alpha}_{\ell})\cong%
	\begin{cases}
		H_k(\mathbb{T}^{2n+1};\mathbb{Z}/2\mathbb{Z}) & \text{if\: $k=n+2,\ldots,2n+1$}, \\
		\left(\mathbb{Z}/2\mathbb{Z}\right)^{\oplus\left(b_k(\mathbb{T}^{2n+1})-b_k(\mathbb{T}^{n+1})\right)} & \text{if\: $k=1,\ldots,n+1$}, \\
		0 & \text{if\: $k=0$}.
	\end{cases}
\]

\end{claim}

\begin{proof}
By the choice of $b$, the set $\mathcal{P}^{[a,b)}(H_T;\hat{\alpha}_{\ell})$ is the product
$(\mathbb{T}^n\setminus e^n)\times\mathbb{T}^{n+1}$.
Theorem \ref{theorem:Pozniak} and Claim \ref{claim:1} show that the Floer homology $\mathrm{HF}_k^{[a,b)}(H_T;\hat{\alpha}_{\ell})$
coincides with the Morse homology
$\mathrm{HM}_{2n+1-k}\bigl(\mathcal{P}^{[a,b)}(H_T;\hat{\alpha}_{\ell}),F_T\bigr)$.
Hence we have
\begin{align*}
	\mathrm{HF}_k^{[a,b)}(\widetilde{H}_T;\hat{\alpha}_{\ell})%
	&\cong \mathrm{HM}_{2n+1-k}\left((\mathbb{T}^n\setminus e^n)\times\mathbb{T}^{n+1},F_T\right)\\
	&\cong \bigoplus_{i+j=2n+1-k} \left(H_i(\mathbb{T}^n\setminus e^n)\otimes_{\mathbb{Z}/2\mathbb{Z}} H_j(\mathbb{T}^{n+1})\right).
\end{align*}
By Poincar\'e(--Lefschetz) duality, we have $H_j(\mathbb{T}^{n+1})\cong H_{n+1-j}(\mathbb{T}^{n+1})$ and
\[
	H_i(\mathbb{T}^n\setminus e^n)%
	\cong H_{n-i}(\mathbb{T}^n\setminus e^n,\partial e^n)%
	\cong\widetilde{H}_{n-i}(\mathbb{T}^n).
\]
Therefore,
\begin{align*}
	\bigoplus_{i+j=2n+1-k} \left(H_i(\mathbb{T}^n\setminus e^n)\otimes_{\mathbb{Z}/2\mathbb{Z}} H_j(\mathbb{T}^{n+1})\right)%
	&\cong\bigoplus_{i+j=k} \left(\widetilde{H}_i(\mathbb{T}^n)\otimes_{\mathbb{Z}/2\mathbb{Z}} H_j(\mathbb{T}^{n+1})\right).
\end{align*}
Since
\begin{align*}
	H_k(\mathbb{T}^{2n+1})
	&\cong\bigoplus_{i+j=k} \left(H_i(\mathbb{T}^n)\otimes_{\mathbb{Z}/2\mathbb{Z}} H_j(\mathbb{T}^{n+1})\right)\\
	&\cong\bigoplus_{\substack{i+j=k\\ i\neq0}} \left(H_i(\mathbb{T}^n)\otimes_{\mathbb{Z}/2\mathbb{Z}} H_j(\mathbb{T}^{n+1})\right)%
	\oplus\bigoplus_{j=k} \left(H_0(\mathbb{T}^n)\otimes_{\mathbb{Z}/2\mathbb{Z}} H_j(\mathbb{T}^{n+1})\right)\\
	&\cong\bigoplus_{i+j=k} \left(\widetilde{H}_i(\mathbb{T}^n)\otimes_{\mathbb{Z}/2\mathbb{Z}} H_j(\mathbb{T}^{n+1})\right)%
	\oplus H_k(\mathbb{T}^{n+1}),
\end{align*}
we conclude that
\[
	\mathrm{HF}_k^{[a,b)}(\widetilde{H}_T;\hat{\alpha}_{\ell})%
	\cong\left(\mathbb{Z}/2\mathbb{Z}\right)^{\oplus\left(b_k(\mathbb{T}^{2n+1})-b_k(\mathbb{T}^{n+1})\right)}.
\]
We note that $\left(\mathbb{Z}/2\mathbb{Z}\right)^{\oplus\left(b_0(\mathbb{T}^{2n+1})-b_0(\mathbb{T}^{n+1})\right)}=0$
and $b_k(\mathbb{T}^{n+1})=0$ for any $k\geq n+2$.
\end{proof}

Thus Claim \ref{claim:6} shows that the homomorphism
\[
	[\iota^F]_k\colon \mathrm{HF}_k^{[a,b)}(\widetilde{H}_T;\hat{\alpha}_{\ell})\to \mathrm{HF}_k^{[a,\infty)}(\widetilde{H}_T;\hat{\alpha}_{\ell})
\]
is not surjective if $R\lvert\ell\rvert<a\leq c-u\ell$ and $k=0,1,\ldots,n+1$.

On the other hand, if $R\lvert\ell\rvert<a\leq c-u\ell$ and $k=0,1,\ldots,n+1$,
by the fundamental homomorphism theorem and the exactness of the first row in \eqref{eq:diagram},
then we have the isomorphism
\[
	\overline{[\pi^F]}_k\colon \mathrm{HF}_k^{[a,\infty)}(\widetilde{H}_T;\hat{\alpha}_{\ell})/\Image [\iota^F]_k\xrightarrow{\cong}%
	\Image [\pi^F]_k.
\]
Since $[\iota^F]_k$ is induced by the inclusion, we obtain
\begin{align*}
	\dim_{\mathbb{Z}/2\mathbb{Z}}\left(\Image [\pi^F]_k\right)%
	&=\dim_{\mathbb{Z}/2\mathbb{Z}}\left(\mathrm{HF}_k^{[a,\infty)}(\widetilde{H}_T;\hat{\alpha}_{\ell})/\Image [\iota^F]_k\right)\\
	&=b_k(\mathbb{T}^{n+1})%
	=\dim_{\mathbb{Z}/2\mathbb{Z}}\left(\mathrm{HF}_k^{[b,\infty)}(\widetilde{H}_T;\hat{\alpha}_{\ell})\right).
\end{align*}
Hence the homomorphism
\[
	[\pi^F]_k\colon \mathrm{HF}_k^{[a,\infty)}(\widetilde{H}_T;\hat{\alpha}_{\ell})\to%
	\mathrm{HF}_k^{[b,\infty)}(\widetilde{H}_T;\hat{\alpha}_{\ell})
\]
is surjective if $R\lvert\ell\rvert<a\leq c-u\ell$ and $k=0,1,\ldots,n+1$.
Thus the proof of Theorem \ref{theorem:main2} is complete.
\end{proof}


\section{Annular capacity}\label{section:4}

To prove Theorem \ref{theorem}, we introduce a homological ``annular'' capacity.
This capacity is defined in terms of the homological Biran--Polterovich--Salamon capacity \cite{BPS}.

\subsection{Homological relative capacity}

Let $(M,\omega)$ be an open symplectic manifold and $A\subset M$ a compact subset.
For any nontrivial free homotopy class $\alpha\in [S^1,M]$ and $c>0$ we define the set
\[
	\boldsymbol{\mathscr{A}}_c(M,A;\alpha)=\left\{\,s\in\mathbb{R}\ (s>0\ \text{if}\ \alpha=0)\relmiddle| T_{\alpha}^{[s,\infty);c}\neq 0\,\right\},
\]
where $T_{\alpha}^{[s,\infty);c}\colon {\SH}^{[s,\infty)}(M;\alpha)\to {\RSH}^{[s,\infty);c}(M,A;\alpha)$ is a homomorphism defined in Subsection \ref{subsection:symp}.

\begin{definition}[{\cite[Subsection 4.9]{BPS}}]
For $\alpha\in [S^1,M]$ and $a\geq -\infty$
we define the \textit{homological Biran--Polterovich--Salamon capacity} $\widehat{C}_{\mathrm{BPS}}(M,A;\alpha,a)\geq 0$ by
\[
	\widehat{C}_{\mathrm{BPS}}(M,A;\alpha,a)=%
	\inf\left\{\, c>0\relmiddle| \sup\boldsymbol{\mathscr{A}}_c(M,A;\alpha)>a\,\right\}.
\]
\end{definition}

Here we use the convention that $\inf\emptyset =\infty$ and $\sup\emptyset =-\infty$.
Let $(N,\omega_N)$ be a closed connected symplectic manifold and $X\subset N$ a compact subset.
Let $A_R$ denote the annulus $(-R,R)\times\mathbb{R}/\mathbb{Z}$,
and for $u\in(-R,R)$ define its subset $L_u=\{u\}\times\mathbb{R}/\mathbb{Z}\subset A_R$.
For $\ell\in\mathbb{Z}$ put
\[
	\alpha_{\ell}=[t\mapsto (0,\ell t)]\in [S^1,A_R].
\]

\begin{definition}
For $R>0$, $u\in(-R,R)$, $\ell\in\mathbb{Z}$ and $a\geq -\infty$,
we define a homological relative capacity $\widehat{C}(N,X;R,u,\ell,a)$ by
\[
	\widehat{C}(N,X;R,u,\ell,a)=\widehat{C}_{\mathrm{BPS}}(A_R\times N,L_u\times X;(\alpha_{\ell},0_N),a).
\]
\end{definition}

Here $A_R\times N$ is considered as the symplectic manifold equipped
with the product symplectic form $(dp_0\wedge dq_0)\oplus\omega_N$ where $(p_0,q_0)\in A_R=(-R,R)\times\mathbb{R}/\mathbb{Z}$.


\subsection{Proof of Theorem \ref{theorem}}\label{subsection:proofofmain}

In the case that $N=\mathbb{T}^{2n}$ and $X=\mathbb{T}^n$,
we can compute the capacity $\widehat{C}$ directly due to Theorem \ref{theorem:main2}.

\begin{proposition}\label{proposition:capacity}
For any $R>0$, $u\in(-R,R)$, $\ell\in\mathbb{Z}$ and $a\geq -\infty$,
\[
	\widehat{C}(\mathbb{T}^{2n},\mathbb{T}^n;R,u,\ell,a)=\max\{R\lvert\ell\rvert+u\ell,a+u\ell\}.
\]
\end{proposition}

\begin{proof}
Let $\hat{\alpha}_{\ell}=(\alpha_{\ell},0_{\mathbb{T}^{2n}})\in [S^1,A_R\times\mathbb{T}^{2n}]$.
By Theorem \ref{theorem:main2}, the homomorphism
\[
	T_{\hat{\alpha}_{\ell}}^{[s,\infty);c}\colon {\SH}^{[s,\infty)}(A_R\times \mathbb{T}^{2n};\hat{\alpha}_{\ell})%
	\to {\RSH}^{[s,\infty);c}(A_R\times \mathbb{T}^{2n},L_u\times \mathbb{T}^n;\hat{\alpha}_{\ell})
\]
is non-zero if and only if $R\lvert\ell\rvert<s\leq c-u\ell$.
Hence we have
\[
	\boldsymbol{\mathscr{A}}_c(A_R\times \mathbb{T}^{2n},L_u\times \mathbb{T}^n;\hat{\alpha}_{\ell})=(R\lvert\ell\rvert,c-u\ell]
\]
for any $c>0$ such that $R\lvert\ell\rvert\leq c-u\ell$.
Therefore, for any $a\geq -\infty$ we have
\begin{align*}
	\widehat{C}(\mathbb{T}^{2n},\mathbb{T}^n;R,u,\ell,a)&=\inf\left\{\,c>0\relmiddle|%
	\sup\boldsymbol{\mathscr{A}}_c(A_R\times \mathbb{T}^{2n},L_u\times \mathbb{T}^n;\hat{\alpha}_{\ell})>a\,\right\}\\%
	&=\max\{R\lvert\ell\rvert+u\ell,a+u\ell\}.\qedhere
\end{align*}
\end{proof}

Proposition \ref{proposition:hat} relates the capacities $\widehat{C}(N,X;R,u,\ell,a)$ and $C(N,X;R,u,\ell,a)$.

\begin{proposition}[{\cite[Proposition 4.9.1]{BPS}}]\label{proposition:hat}
Let $R>0$ and $u$ be real numbers such that $u\in(-R,R)$.
Let $\ell$ be an integer and $a$ a real number.
If $\widehat{C}(N,X;R,u,\ell,a)<\infty$
then every Hamiltonian $H$ with compact support on $S^1\times A_R\times N$
with $H|_{S^1\times L_u\times X}\geq \widehat{C}(N,X;R,u,\ell,a)$ has
a one-periodic trajectory $x$ in the homotopy class $(\alpha_{\ell},0_N)$ with the action $\mathcal{A}_H(x)\geq a$.
In particular,
\[
	\widehat{C}(N,X;R,u,\ell,a)\geq C(N,X;R,u,\ell,a).
\]
\end{proposition}

\begin{proof}[Proof of Theorem \ref{theorem:main3}]
By Proposition \ref{proposition:capacity} and Proposition \ref{proposition:hat}, it is enough to show that
\[
	C(\mathbb{T}^{2n},\mathbb{T}^n;R,u,\ell,a)\geq \max\{R\lvert\ell\rvert+u\ell,a+u\ell\}.
\]
If $\ell=0$ and $a\leq 0$,
then every Hamiltonian $H$ with compact support has a contractible periodic trajectory whose action is zero
and hence $C(\mathbb{T}^{2n},\mathbb{T}^n;R,u,0,a)=0$.
Thus we assume that either $\ell=0$ and $a>0$, or $\ell\neq 0$.
We set $m=\max\{R\lvert\ell\rvert+u\ell,a+u\ell\}$.
For any $\delta>0$ choose a smooth function $f\colon (-R,R)\to\mathbb{R}$ with compact support satisfying
\[
	f(r)=%
	\begin{cases}
		m-\delta & \text{for}\ r\ \text{near}\ u, \\
		0 & \text{for}\ r\ \text{near}\ \pm R,
	\end{cases}
\]
\[
	f(r)<%
	\begin{cases}
		\frac{R+r}{R+u}m & \text{if\ $-R<r\leq u$},\\
		\frac{R-r}{R-u}m & \text{if\ $u\leq r<R$},
	\end{cases}
\]
and
\[
	f'(r)
	\begin{cases}
	<\frac{m}{R+u} & \text{if\ $-R<r\leq u$},\\
	>-\frac{m}{R-u} & \text{if\ $u\leq r<R$}.
	\end{cases}
\]
Now consider the Hamiltonian $H\colon A_R\times \mathbb{T}^{2n}\to\mathbb{R}$ with compact support given by
$H(p,q)=f(p_0)$.
Then every periodic trajectory $x\colon t\mapsto x(t)=\bigl(p(t),q(t)\bigr)$
of $H$ in the homotopy class $(\alpha_{\ell},0_{\mathbb{T}^{2n}})\in [S^1,A_R\times\mathbb{T}^{2n}]$ satisfies
\[
	\dot{q}_0\equiv\ell,\ \dot{q}_1\equiv\cdots\equiv \dot{q}_n\equiv 0,\ %
	p_0(t)=r,\ \dot{p}_1\equiv\cdots\equiv\dot{p}_n\equiv 0,
\]
where $r\in(-R,R)$ is such that $f'(r)=\ell$.
Moreover, the action of $x$ is $\mathcal{A}_H(x)=f(r)-r\ell$.
If $R\lvert\ell\rvert\geq a$,
then we have $m=R\lvert\ell\rvert+u\ell$ and hence
\[
	f'(r)
	\begin{cases}
	<\frac{R\lvert\ell\rvert+u\ell}{R+u}=\ell & \text{if\ $-R<r\leq u$},\\
	>-\frac{R\lvert\ell\rvert+u\ell}{R-u}=\ell & \text{if\ $u\leq r<R$}.
	\end{cases}
\]
Here we note that $\ell\geq 0$ (resp.\ $\ell<0$) if and only if $r\leq u$ (resp.\ $r>u$).
The above observation contradicts the fact that $f'(r)=\ell$.
Hence there is no one-periodic trajectory of length $\lvert\ell\rvert$.
We assume that $R\lvert\ell\rvert< a$.
Then $m=a+u\ell$.
If $\ell\geq 0$, then
\begin{align*}
	\mathcal{A}_H(x)&=f(r)-r\ell<\frac{R+r}{R+u}m-r\ell=\frac{(R+r)a+(u-r)R\ell}{R+u}\\
	&<\frac{(R+r)a+(u-r)a}{R+u}=a.
\end{align*}
If $\ell<0$, then
\begin{align*}
	\mathcal{A}_H(x)&=f(r)-r\ell<\frac{R-r}{R-u}m-r\ell=\frac{(R-r)a-(r-u)R\ell}{R-u}\\
	&=\frac{(R-r)a+(r-u)R\lvert\ell\rvert}{R-u}<\frac{(R-r)a+(r-u)a}{R-u}=a.
\end{align*}
Thus there is no periodic trajectory in $(\alpha_{\ell},0_{\mathbb{T}^{2n}})$ whose action is at least $a$.

As a conclusion, we obtain that
\[
	C(\mathbb{T}^{2n},\mathbb{T}^n;R,u,\ell,a)=C_{\mathrm{BPS}}(A_R\times \mathbb{T}^{2n},L_u\times \mathbb{T}^n;(\alpha_\ell,0_{\mathbb{T}^{2n}}),a)%
	\geq m-\delta
\]
for any $\delta>0$.
It means that
\[
	C(\mathbb{T}^{2n},\mathbb{T}^n;R,u,\ell,a)\geq m.\qedhere
\]
\end{proof}

Finally we prove Theorem \ref{theorem}.

\begin{proof}[Proof of Theorem \ref{theorem}]
By \cite[Proof of Theorem B]{BPS}, we may assume, without loss of generality,
that $H$ is one-periodic in time.
If $\ell=0$ and $a\leq 0$,
then every Hamiltonian $H$ with compact support has infinitely many contractible periodic trajectories whose actions are zero,
and hence Theorem \ref{theorem} is proved.
Thus we assume that either $\ell=0$ and $a>0$, or $\ell\neq 0$.
According to Theorem \ref{theorem:main3}, we have
\[
	C(\mathbb{T}^{2n},\mathbb{T}^n;R,u,\ell,c-u\ell)=\max\{R\lvert\ell\rvert+u\ell,c\}=c.
\]
It implies that for all $H\in\mathcal{H}_c(A_R\times \mathbb{T}^{2n},L_u\times \mathbb{T}^n)$
there exists $x\in\mathcal{P}(H;(\alpha_{\ell},0_{\mathbb{T}^{2n}}))$ such that $\mathcal{A}_H(x)\geq c-u\ell$.
Moreover, by Theorem \ref{theorem:main2},
if $H$ is non-degenerate, then the number of such $x$'s is at least
\[
	\sum_{k=0}^{n+1}b_k(\mathbb{T}^{n+1})=b_{\mathbb{T}^{n+1}}.\qedhere
\]
\end{proof}


\subsection*{Acknowledgement}

The authors would like to express their sincere gratitude to their advisor Professor Takashi Tsuboi for his practical advice
and to the referee for very important remarks.
They also thank Satoshi Sugiyama for giving them the trigger to study the present topic.
They thank Hiroyuki Ishiguro for pointing out many careless mistakes in a previous draft, too.
The second author thanks Natsumi Magome for drawing beautiful 3D graphics.
This work was supported by IBS-R003-D1 (the first author), JSPS KAKENHI Grant Numbers 25-6631 (the first author), 26-7057 (the second author)
and the Program for Leading Graduate Schools, MEXT, Japan.
The authors were supported by the Grant-in-Aid for JSPS fellows.


\bibliographystyle{amsart}

\begin{thebibliography}{10}
\bibitem[BPS]{BPS} P. Biran, L. Polterovich and D. Salamon,
	{\it Propagation in Hamiltonian dynamics and relative symplectic homology},
	Duke Math.\ J. \textbf{119} (2003), no.~1, 65--118.
\bibitem[Ci]{Ci} K. Cieliebak,
	{\it Handle attaching in symplectic homology and the chord conjecture},
	J. Eur.\ Math.\ Soc.\ (JEMS) \textbf{4} (2002), 115--142.
\bibitem[CFH]{CFH} K. Cieliebak, A. Floer and H. Hofer,
	{\it Symplectic homology, II.\ A general construction},
	Math.\ Zeit.\ \textbf{218} (1995), 103--122.
\bibitem[Fl]{Fl} A. Floer,
	{\it Symplectic fixed points and holomorphic spheres},
	Comm.\ Math.\ Phys.\ \textbf{120} (1989), 575--611.
\bibitem[FH]{FH} A. Floer and H. Hofer,
	{\it Symplectic homology, I.\ Open sets in $\mathbb{C}^n$},
	Math.\ Zeit.\ \textbf{215} (1994), 37--88.
\bibitem[FHS]{FHS} A. Floer, H. Hofer and D. Salamon,
	{\it Transversality in elliptic Morse theory for the symplectic action},
	Duke Math.\ J. \textbf{80} (1995), 251--292.
\bibitem[FS]{FS} U. Frauenfelder and F. Schlenk,
	{\it Hamiltonian dynamics on convex symplectic manifolds},
	Israel J.\ Math.\ \textbf{159} (2007), 1--56.
\bibitem[Is]{Is} H. Ishiguro,
	{\it Non-contractible orbits for Hamiltonian functions on Riemann surfaces},
	ArXiv e-print 1612.07062 (2016).
\bibitem[Ka]{Ka} M. Kawasaki,
	{\it Heavy subsets and non-contractible trajectories},
	ArXiv e-print 1606.01964 (2016).
\bibitem[Ni]{Ni} C. Niche,
	{\it Non-contractible periodic orbits of Hamiltonian flows on twisted cotangent bundles},
	Discrete Contin.\ Dyn.\ Syst., \textbf{14} (4) (2006), 617--630.
\bibitem[Po]{Po} M. Po\'{z}niak,
	{\it Floer homology, Novikov rings and clean intersections},
	in Northern California Symplectic Geometry Seminar, ed.\ Y. Eliashberg, D. Fuchs, T. Ratiu,
	and A. Weinstein, Amer.\ Math.\ Soc.\ Transl.\ Ser.\ 2 \textbf{196}, Amer.\ Math.\ Soc., Providence, 1999, 119--181.
\bibitem[Sa]{Sa} D. Salamon,
	{\it Lectures on Floer homology},
	in Symplectic Geometry and Topology (Park City, Utah, 1997),
	IAS/Park City Math.\ Ser.\ \textbf{7}, Amer.\ Math.\ Soc., Providence, 1999, 143--230.
\bibitem[SZ]{SZ} D. Salamon and E. Zehnder,
	{\it Morse theory for periodic solutions of Hamiltonian systems and the Maslov index},
	Comm.\ Pure Appl.\ Math.\ \textbf{45} (1992), 1303--1360.
\bibitem[Us]{Us} M. Usher,
	{\it The sharp energy-capacity inequality},
	Commun.\ Contemp.\ Math.\ \textbf{12} (2010), no.~3, 457--473. 
\bibitem[Vi]{Vi} C. Viterbo,
	{\it Functors and computations in Floer homology with applications I},
	Geom.\ funct.\ anal.\ \textbf{9} (1999), 985--1033.
\bibitem[We]{We} J. Weber,
	{\it Noncontractible periodic orbits in cotangent bundles and Floer homology},
	Duke Math.\ J. \textbf{133} (2006), no.~3, 527--568.
\bibitem[Xu]{Xu} J. Xue,
	{\it Existence of noncontractible periodic orbits of Hamiltonian system separating two Lagrangian tori on $T^{\ast}\mathbb{T}^{n}$ with application to non convex Hamiltonian systems},
	ArXiv e-print 1408.5193 (2014); to appear in J. Symplectic Geom.
\end{thebibliography}

\end{document}